\documentclass[preprint]{elsarticle}

\usepackage{graphicx}
\usepackage{graphics}
\usepackage{epsfig}

\usepackage{amsmath,amssymb,amscd,amsthm}

\usepackage{caption}
\usepackage{subcaption}

\usepackage{xcolor}
\usepackage{soul}

\usepackage{tikz}

\usepackage{tcolorbox}
\usepackage{multicol}
\usepackage{parcolumns}
\newtheorem{theorem}{Theorem}[section]
\newtheorem{definition}[theorem]{Definition}
\newtheorem{lemma}[theorem]{Lemma}

\newtheorem{question}[theorem]{Question}

\begin{document} 

\title{Presentations and Representations of the Multi-Virtual Twin Group and Associated Subgroups}

\author[1]{Vaibhav Keshari\corref{cor1}}
\ead{vaibhav.23maz0022@iitrpr.ac.in}

\author[2]{Taher I. Mayassi}
\ead{taher.mayasi@liu.edu.lb}

\author[1]{Madeti Prabhakar}
\ead{prabhakar@iitrpr.ac.in}

\author[3]{Mohamad N. Nasser}
\ead{m.nasser@bau.edu.lb}

\cortext[cor1]{Corresponding author}

\affiliation[1]{
organization={Department of Mathematics, Indian Institute of Technology Ropar},
city={Rupnagar},
state={Punjab},
postcode={140001},
country={India}
}

\affiliation[2]{
organization={Department of Mathematics and Physics, Lebanese International University},
city={Beirut},
country={Lebanon}
}

\affiliation[3]{
organization={Department of Mathematics and Computer Science, Beirut Arab University},
city={Beirut},
country={Lebanon}
}




\begin{abstract}
Motivated by the notion of the multi-virtual braid group introduced by L. Kauffman, and by the study of extensions of the well-known twin group $T_n, n\geq 2$, we introduce a new group called the multi-virtual twin group $M_kVT_n$, where $k \geq 1$ and $n \geq 2$, together with two associated subgroups: the multi-virtual pure twin group $M_kVPT_n$ and the multi-virtual semi-pure twin group $M_kVHT_n$. We then classify all homogeneous $2$-local representations of $M_kVT_n$ into $\mathrm{GL}_n(\mathbb{C})$ for all $k\geq 1$ and $n\geq 3$, and show that they fall into exactly eight distinct types. We also investigate their main properties, including faithfulness and irreducibility, proving that they are generally unfaithful and providing necessary and sufficient conditions for their irreducibility. Furthermore, for certain values of $k$ and $n$, we construct non-local representations of $M_kVPT_n$ induced from those of $M_kVT_n$, and we determine the conditions under which these induced representations are irreducible. Finally, we present several problems for future research in this area.
\end{abstract}
\maketitle

\section{Introduction}
The twin group $T_n$, for $n \geq 2$, is a well-known class of right-angled Coxeter groups generated by $n-1$ involutions denoted as $s_1, s_2, \dots, s_{n-1}$ satisfying the far commutation relations: $s_is_j = s_js_i$ for $|i-j| \geq 2$. This generating set is referred to as the twin generators. This group was first introduced by Shabat and Voevodsky~\cite{SV}, and later studied in detail by Khovanov~\cite{M.Khovanov1996, MK}, who provided a geometric interpretation analogous to that of the classical braid group $B_n$. Twin groups, together with their subgroups and extensions, naturally arise in the study of planar configurations of arcs and play an important role in the theory of doodles on surfaces. For more information on the twin group, see \cite{Merkov1999,BSV,Mosto2020,TNM11}.\vspace{0.1cm}

The twin group $T_n$ admits several important subgroups and extensions that have been widely studied in the literature. An important subgroup of $T_n$, introduced in~\cite{BSV}, is the pure twin group $PT_n$, defined as the kernel of the natural homomorphism from $T_n$ onto the symmetric group $S_n$. This group can be viewed as a planar analogue of the pure braid group $P_n$. Among the various extensions of $T_n$, the virtual twin group $VT_n$ has been extensively studied \cite{BSV}. It is defined in analogy with the virtual braid group $VB_n$ and is generated by the twin generators $s_1, s_2, \dots, s_{n-1}$ inherited from $T_n$, together with an additional family of generators $\rho_1, \rho_2, \dots, \rho_{n-1}$ representing virtual interactions. Its pure subgroup, the pure virtual twin group $PVT_n$, is defined as the kernel of the natural projection from $VT_n$ onto $S_n$. Another important extension of $T_n$ is the singular twin group $ST_n$, introduced by Nasser and Chbili~\cite{NasserNafaa}, obtained by adjoining a new family of generators $\tau_1, \tau_2, \ldots, \tau_{n-1}$ to the twin generators of $T_n$, with defining relations analogous to those of the singular braid group $SB_n$. Subsequently, Caprau and Nasser introduced a further generalization of both $T_n$ and $ST_n$ by adding a third family of generators $\nu_1, \nu_2, \ldots, \nu_{n-1}$, which leads to the virtual singular twin group $VST_n$~\cite{Caprau}. \vspace{0.1cm}

In~\cite{LHK2025}, Kauffman introduced the multi-virtual braid group as a generalization of the virtual braid group by allowing multiple families of virtual generators. This framework is naturally connected to multi-virtual knot theory, where several types of virtual crossings are considered. In this context, each family of virtual generators corresponds to a different type of virtual interaction, leading to a more intricate algebraic and topological structure. Motivated by Kauffman’s construction, together with the various extensions of the twin group developed by several authors mentioned in the previous paragraph, we are led to define in our work a multi-virtual version of the twin group, which will be an extension of the twin group $T_n$ and a generalization of the virtual twin group $VT_n$.\vspace{0.1cm}

Representation theory provides a powerful framework for studying algebraic structures by realizing group elements as linear transformations or automorphisms, which helps reveal their underlying structural features. In this context, $h$-local matrix representations~\cite{Nas1} are particularly effective, since each generator acts non-trivially only on a small part of the matrix while acting as the identity elsewhere. This local behavior makes them especially suitable for the study of braid-related and Coxeter-type groups. The theory of $h$-local representations for braid groups and their multi-virtual extensions is well established~\cite{Mik, chrr, Vk2025}, as well as for the twin groups~\cite{Mayasi20251, Nasser202622}. In particular, properties such as faithfulness and irreducibility remain central, as they capture key aspects of the underlying algebraic structure. \vspace{0.1cm}

Motivated by these developments in the literature, we introduce in Section 3 a new family of groups called the multi-virtual twin groups and denoted by $M_kVT_n$, where $k\geq 1$ and $n\geq 2$, together with their natural subgroups, namely the multi-virtual pure twin group $M_kVPT_n$ and the multi-virtual semi-pure twin group $M_kVHT_n$. These groups extend the framework of virtual twin groups by allowing multiple layers of virtual interactions, thereby yielding a more flexible algebraic setting. In Section 4, we classify all homogeneous $2$-local representation of $M_kVT_n$ into $\mathrm{GL}_n(\mathbb{C})$ for all $k\geq 1$ and $n\geq 3$. We also study their main properties, including faithfulness and irreducibility. In Section 5, we focus on representations of $M_kVPT_n$ for particular values of $k$ and $n$, using the representations developed in Section 4, and establish conditions under which these representations are irreducible. Finally, in Section 6, we present several questions for further investigation.

\section{Overview on Twin Groups and Their $h$-Local Representations}

\subsection{Twin Groups}
The twin group $T_n$, where $n\geq 2,$ is given by the following presentation:
\[
T_n = \langle s_1,s_2,\dots, s_{n-1} ~|~ s_i^2 =1 \text{ for }  1\leq i \leq n-1, s_is_j = s_js_i \text{ for } |i-j|\geq2 \rangle.
\]
\noindent For small values of \(n\), the twin group \(T_n\) can be described explicitly as follows:
\begin{itemize}
\item \(T_2=\langle s_1 \mid s_1^2=1 \rangle \cong \mathbb{Z}_2\), which is the cyclic group of order \(2\).
\item \(T_3=\langle s_1,s_2 \mid s_1^2=s_2^2=1 \rangle \cong \mathbb{Z}_2 * \mathbb{Z}_2\), namely the infinite dihedral group.
\end{itemize}
Each generator $s_i$ can be presented diagrammatically using the configuration illustrated in Fig.~\ref{fig:generator} $(a)$.

\begin{figure}
    \centering
     \begin{tikzpicture}[scale=1]
\draw[thick] (-2.7,-1) -- (-2.7,1);
\draw[thick,dashed] (-2.6,0)--(-2.1,0);
\draw[thick] (-2,-1) -- (-2,1);
  \draw[thick] (-1.5,-1) -- (0.5,1);
  \draw[thick] (-1.5,1) -- (0.5,-1);
  \draw[thick] (1,-1) -- (1,1);
\draw[thick,dashed] (1.1,0)--(1.6,0);
\draw[thick] (1.7,-1) -- (1.7,1);
  \draw[thick] (3.5,-1) -- (3.5,1);
\draw[thick,dashed] (3.6,0)--(4.1,0);
\draw[thick] (4.2,-1) -- (4.2,1);
  \draw[thick] (4.7,-1) -- (6.7,1);
  \draw[thick] (4.7,1) -- (6.7,-1);
  \foreach \xy in {(5.7,0),(5.7,0)} 
    \draw[thick] \xy circle(0.1);
    \draw[thick] (7.2,-1) -- (7.2,1);
\draw[thick,dashed] (7.3,0)--(7.8,0);
\draw[thick] (7.9,-1) -- (7.9,1);
    \node at (-2.7,1.2) {$1$};
 \node at (-2.1,1.2) {$i-1$};
 \node at (-1.5,1.2) {$i$};
   \node at (0.35,1.2) {$i+1$};  
    \node at (1.15,1.2) {$i+2$};
     \node at (1.75,1.2) {$n$};
       \node at (3.4,1.2) {$1$};
 \node at (4.1,1.2) {$i-1$};
 \node at (4.6,1.2) {$i$};
   \node at (6.5,1.2) {$i+1$};  
    \node at (7.35,1.2) {$i+2$};
     \node at (8,1.2) {$n$};
     \node at (5.8,-1.5) {$(b)$~$\rho_i$};
     \node at (0,-1.5) {$(a)~ s_i$};
\end{tikzpicture}
    \caption{The Generators $s_i$ and $\rho_i$}
    \label{fig:generator}
\end{figure}
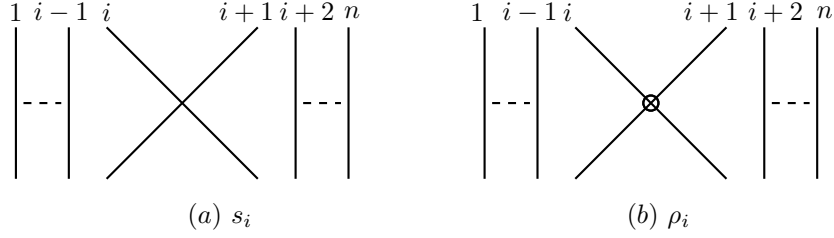
The pure twin group $PT_n$, where $n\geq 2,$ is defined as the kernel of the natural surjective homomorphism $ \pi: T_n \longrightarrow S_n$ defined by $ \pi(s_i) = (i\ \ i+1) \text{ for } i=1,2,\dots, n-1.$ In \cite{BSV}, it has found a generating set of $PT_n$ for $n\geq 3$ using the Reidemeister-Schreier
 method. In addition, we can see that the group $T_n$ admits the decomposition 
\[
T_n \cong PT_n \rtimes S_n.
\]

The virtual twin group $VT_n$, where $n\geq 2$, extends the twin group $T_n$ by incorporating an additional family of generators. It can be described by the generators \{$s_1,s_2, \dots,s_{n-1},\rho_1,\rho_2,\dots,\rho_{n-1} $\} and the following relations:
\begin{align}
    s_i^2 &= 1 \text{  for } i =1,2,\dots ,n-1,\\
    s_is_j &= s_js_i \text{ for } |i-j|\geq 2,\\
    \rho_i^2 &= 1 \text{ for } i=1,2,\dots,n-1,\\
    \rho_i \rho_j &= \rho_j \rho_i \text{ for } |i-j|\geq 2,\\
    \rho_i\rho_{i+1}\rho_i &= \rho_{i+1} \rho_i \rho_{i+1} \text{ for } i=1,2,\dots,n-2,\\
    \rho_i s_j &= s_j\rho_i \text{ for } |i-j|\geq 2,\\
    \rho_i \rho_{i+1} s_i &= s_{i+1} \rho_i \rho_{i+1} \text{ for } i=1,2,\dots n-2,
\end{align}
where the generator $\rho_i$ is presented in a diagrammatic way using the configuration illustrated in  Fig.~\ref{fig:generator} $(b)$. On the other hand, geometrically, the virtual twin group can be understood through the following construction. Consider the known Euclidean plane $\mathbb{R}^2$ =$\{(x,y) ~|~ x,y \in \mathbb{R} \}$ with two parallel lines $y=0$ and $y=1$ on it. Let \((1,1), \ldots, (n,1)\) and \((1,0), \ldots, (n,0)\) be $2n$ marked points on these lines. A virtual twin diagram on $n$ strands is a configuration of $n$ arcs in $\mathbb{R} \times [0,1]$ that connect points $(1,1), \dots,(n,1)$ with points $(1,0),\dots,(n,0)$ in some order and satisfying the following conditions:
\begin{itemize}
    \item The projection of each arc on the $y$-coordinate is a homeomorphism onto $[0,1]$, that is, the arcs are monotonic.
   \item The set \( V(D) \) of crossings of a diagram \( D \) consists only of transverse double points, each equipped with the additional structure of being either a real or a virtual crossing, as illustrated in Fig.~2.
\end{itemize}
\begin{figure}
    \centering
    \begin{tikzpicture}[scale =1]
        \draw (0,0) -- (2,2);
        \draw (0,2) --(2,0);
        \draw (4,0) -- (6,2);
        \draw (4,2) -- (6,0);
        \draw (5,1) circle (5pt);
    \end{tikzpicture}
    \caption{Real and Virtual Crossings}
    \label{fig:placeholder}
\end{figure}
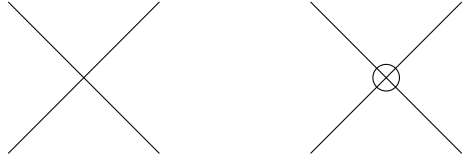

In a similar way of the construction of the pure twin group, the pure virtual twin group $PVT_n$ is defined as the kernel of the natural surjective homomorphism $ \pi: VT_n \longrightarrow S_n$ defined by $ \pi(s_i) =\pi(\rho_i) = (i\ \ i+1) \text{ for } i=1,2,\dots, n-1.$ It is known that $PVT_n$ admits a presentation with the generators $\lambda _{i,j}, 1\leq i < j \leq n,$ and the following defining relations: 
\begin{align}
     \lambda_{i,j} \lambda _{k,l} = \lambda_{k,l} \lambda_{i,j}
\end{align}
for all distinct integers $i, j, k,l$.
It was shown in \cite{TNM} that the generators of $PVT_n$ can be expressed in terms of the generators of $VT_n$ as follows:
\begin{align}
\lambda_{i,i+1} & = s_i\rho_i \text{  for  } 1\leq i\leq n-1,\\
\lambda _{i,j} &= \rho_{j-1}\rho_{j-2} \dots \rho_{i+1}\lambda_{i,i+1} \rho_{i+1} \dots \rho_{j-2} \rho_{j-1} \text{ for } 1\leq i<j\leq n, j\neq i+1.
\end{align}
Additionally, we can see that the group $VT_n$ admits the decomposition 
\[
VT_n \cong PVT_n \rtimes S_n.
\]

\subsection{The $h$-Local Representations}
\noindent This subsection introduces \(h\)-local representations of finitely generated groups and recalls known results of $h$-local representations of the braid and twin groups.

\begin{definition} \cite{Nas1}
Let $G$ be a group with generators $g_1,g_2,\ldots,g_{n-1}$. A representation $\theta: G \longrightarrow \mathrm{GL}_{m}(\mathbb{C})$ is said to be $h$-local if it is of the form
$$\theta(g_i) =\left( \begin{array}{c|@{}c|c@{}}
   \begin{matrix}
     I_{i-1} 
   \end{matrix} 
      & \textbf{0} & \textbf{0} \\
      \hline
    \textbf{0} &\hspace{0.2cm} \begin{matrix}
   		M_i
   		\end{matrix}  & \textbf{0}  \\
\hline
\textbf{0} & \textbf{0} & I_{n-i-1}
\end{array} \right) \hspace*{0.2cm} \text{for} \hspace*{0.2cm} 1\leq i\leq n-1,$$ 
where $M_i \in \mathrm{GL}_h(\mathbb{C})$ with $h=m-n+2$ and $I_r$ is the $r\times r$ identity matrix. The representation $\theta$ is called homogeneous if the matrices \(M_i, 1 \leq i \leq n-1\), are equal.
\end{definition}
Observe that the notion of $h$-local representations naturally extends to any group $G$ generated by $r(n-1)$ elements, where the generators are partitioned into $r$ distinct families, each consisting of $n-1$ elements. For clarity, we present the construction in the special case $r = 2$ in the definition below.

\begin{definition}
Let $G$ be a group with two families of generators $g_1,g_2,\ldots,g_{n-1}$ and $g_1',g_2',\ldots,g_{n-1}'$. An $h$-local representation $\theta: G \longrightarrow \mathrm{GL}_{m}(\mathbb{C})$ is a representation of the form
$$\theta(g_i) =\left( \begin{array}{c|@{}c|c@{}}
   \begin{matrix}
     I_{i-1} 
   \end{matrix} 
      & \textbf{0} & \textbf{0} \\
      \hline
    \textbf{0} &\hspace{0.2cm} \begin{matrix}
   		M_i
   		\end{matrix}  & \textbf{0}  \\
\hline
\textbf{0} & \textbf{0} & I_{n-i-1}
\end{array} \right)
\text{ and } \theta(g'_i) =\left( \begin{array}{c|@{}c|c@{}}
   \begin{matrix}
     I_{i-1} 
   \end{matrix} 
      & \textbf{0} & \textbf{0} \\
      \hline
    \textbf{0} &\hspace{0.2cm} \begin{matrix}
   		M_i'
   		\end{matrix}  & \textbf{0}  \\
\hline
\textbf{0} & \textbf{0} & I_{n-i-1}
\end{array} \right) \vspace{0.5cm} $$
for $1\leq i\leq n-1,$ where $M_i,M_i' \in \mathrm{GL}_h(\mathbb{C})$ with $h=m-n+2$ and $I_r$ is the $r\times r$ identity matrix. In this case, $\theta$ is homogeneous if all the matrices \(M_i, 1 \leq i \leq n-1\), are equal and all the matrices \(M'_i, 1 \leq i \leq n-1\), are equal.
\end{definition}

With regard to $h$-local representations of the twin group $T_n$, Mayassi and Nasser classified all complex homogeneous $2$-local representations of $T_n$ for every $n \geq 2$ and provided a complete analysis of their irreducibility \cite{Mayasi20251}. In addition, Nasser determined all complex homogeneous $3$-local representations of the twin group $T_n$, as well as those of the virtual twin group $VT_n$ and the welded twin group $WT_n$, for all $n \geq 4$, and investigated their faithfulness and irreducibility \cite{Nasser202622}. These results naturally suggest extending the study to broader families of twin-type groups. This motivates us to classify and study the $h$-local representations of our targeted group in this paper, the multi-virtual twin group, which will be defined explicitly in the next section.

\section{Multi-Virtual Twin Group}
 This section begins by introducing  the multi-virtual twin group \(M_kVT_n\), a generalization of the virtual twin group with multiple types of virtual crossings. For $k \geq 1$ and $n \geq 2$, the multi-virtual twin group is the group generated by the elements \[s_i \text{ and }\rho_i^{\alpha},\text{ where } i=1,2 \dots ,n-1, \alpha=0,1,\dots, k-1,\] 
with the defining relations given as follows:
\begin{align}
    s_i^2 &=1 \text{ for } i=1,2,\dots,n-1,\\
    s_is_j &= s_js_i \text{ for } |i-j|\geq 2,\\
    (\rho_i^{\alpha})^2 &= 1 \text{ for } i=1,2, \dots ,n-1 \text{ and } \alpha= 0,1,\dots, k-1 ,\\
    \rho_i^{\alpha} \rho_j^{\beta} &= \rho_j^{\beta} \rho_i^{\alpha} \text{ for } |i-j|\geq 2 \text{ and } 0\leq \alpha, \beta\leq k-1 ,\\
    \rho_i^{\alpha}s_j &= s_j  \rho_i^{\alpha} \text{ for } |i-j|\geq2 \text{ and }  \alpha= 0,1,\dots,k-1,\\
    \rho_i^{\alpha} \rho_{i+1}^{\alpha} \rho_i^{\alpha} &= \rho_{i+1}^{\alpha} \rho_i^{\alpha} \rho_{i+1}^{\alpha} \text{ for } i=1,2,\dots,n-2 \text{ and } \alpha = 0,1,\dots,k-1 ,\\
    \rho_i^{\alpha} \rho_{i+1}^{\beta} \rho_i^{\beta} &= \rho_{i+1}^{\beta} \rho_i^{\beta} \rho_{i+1}^{\alpha} \text{ for } i=1,2,\dots,n-2 \text{ and } 0 \leq \alpha < \beta \leq k-1,   \\
    \rho_i^{\alpha} \rho_{i+1}^{\alpha} \rho_i^{\beta} &= \rho_{i+1}^{\beta} \rho_i^{\alpha} \rho_{i+1}^{\alpha} \text{ for } i=1,2,\dots,n-2 \text{ and } 0 \leq \alpha < \beta \leq k-1, \\
     \rho_i^{\alpha} \rho_{i+1}^{\alpha} s_i &= s_{i+1} \rho_i^{\alpha} \rho_{i+1}^{\alpha} \text{ for } i=1,2,\dots,n-2 \text{ and } \alpha= 0,1,\dots,k-1.
    \end{align}
The diagrammatic interpretations of $s_i$ and $\rho_i^{\alpha}$ are depicted in Fig.~\ref{fig:sipi}. 
    \begin{figure}
        \centering
       \begin{tikzpicture}[scale=1]
\begin{scope}
\draw[thick] (-2.7,-1) -- (-2.7,1);
\draw[thick,dashed] (-2.6,0)--(-2.1,0);
\draw[thick] (-2,-1) -- (-2,1);
  \draw[thick] (-1.5,-1) -- (0.5,1);
  \draw[thick] (-1.5,1) -- (0.5,-1);
  \draw[thick] (1,-1) -- (1,1);
\draw[thick,dashed] (1.1,0)--(1.6,0);
\draw[thick] (1.7,-1) -- (1.7,1);
  \draw[thick] (2.5,-1) -- (2.5,1);
\draw[thick,dashed] (2.6,0)--(3.1,0);
\draw[thick] (3.2,-1) -- (3.2,1);
  \draw[thick] (3.7,-1) -- (5.7,1);
  \draw[thick] (3.7,1) -- (5.7,-1);
  \foreach \xy in {(4.7,0),(4.7,0)} 
    \draw[thick] \xy circle(0.1);
    \draw[thick] (6.2,-1) -- (6.2,1);
\draw[thick,dashed] (6.3,0)--(6.8,0);
\draw[thick] (6.9,-1) -- (6.9,1);
    \node at (-2.7,1.2) {$1$};
 \node at (-2.1,1.2) {$i-1$};
 \node at (-1.5,1.2) {$i$};
   \node at (0.35,1.2) {$i+1$};  
    \node at (1.15,1.2) {$i+2$};
     \node at (1.75,1.2) {$n$};
       \node at (2.4,1.2) {$1$};
 \node at (3.1,1.2) {$i-1$};
 \node at (3.6,1.2) {$i$};
   \node at (5.5,1.2) {$i+1$};  
    \node at (6.35,1.2) {$i+2$};
     \node at (7,1.2) {$n$};
    \node at (5,0) {$\alpha$};
        
\end{scope}
\end{tikzpicture}
        \caption{The Generators $s_i$ and $\rho_i^{\alpha}$}
        \label{fig:sipi}
    \end{figure}
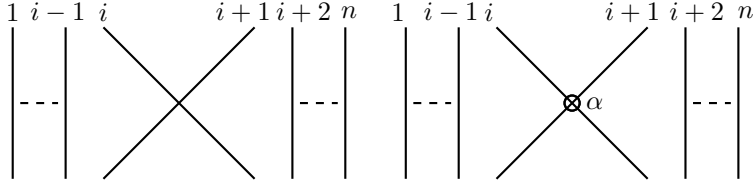
    
\noindent To distinguish between the different types of virtual crossings, we present in Fig.~\ref{fig:example} an example of the element $s_1 \rho_2^{\alpha} \rho_3^{\beta} s_1$ in $M_2VT_4$.
\begin{figure}
    \centering
    \begin{tikzpicture} [scale=0.9]
        \draw[thick] (0,0) -- (5,0);
        \draw[thick] (0,3) -- (5,3);
        \draw[thick] (1,3) to[in=90, out=270] (4,0); 
        \draw[thick] (2,3) .. controls (2,2) and (1,1.5) .. (2,0);
        \draw[thick] (3,3) .. controls (3.5,1.5) and (1.5,0.5) .. (1,0);
        \draw[thick] (4,3) to[in=90 , out=290 ] (3,0); 
        \draw[thick] (2.63,1.43) circle (5pt);
        \draw[thick] (3.35,0.85) -- (3.55,1.05) -- (3.35,1.25) -- (3.15,1.05) -- (3.35,0.85);
        \node[thick] at (2.5,1) {$\alpha$};
        \node[thick] at (3.4,0.6) {$\beta$};
    \end{tikzpicture}
    \caption{The Multi-Virtual Twin $s_1 \rho_2^{\alpha} \rho_3^{\beta} s_1$}
    \label{fig:example}
\end{figure}
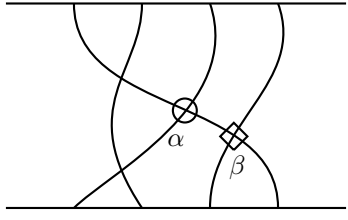

\subsection{Multi-Virtual Pure Twin Group}
In this subsection, we define the multi-virtual pure twin group $M_kVPT_n$, in analogy with the pure virtual twin group. For $k \geq 1$ and $n \geq 2$, define the map $$\Phi_{n,k}: M_kVT_n \longrightarrow S_n $$ given by $$ \Phi_{n,k}(s_i)=\Phi_{n,k}(\rho_i^{\alpha})=(i\ \ i+1)$$ for $i=1,2,\dots,n-1$ and $\alpha = 0,1,\dots,k-1.$ We introduce the definition of the multi-virtual pure twin group to be the kernel of the map $\Phi_{n,k}$. We have that $M_kVPT_n$ is a normal subgroup of $M_kVT_n$, and we can obtain the following short exact sequence
\[
1 \longrightarrow M_kVPT_n \longrightarrow M_kVT_n \longrightarrow S_n \longrightarrow 1.
\] Now, define the map $\theta:S_n \longrightarrow M_kVT_n$ on the generators of $S_n$ by \[ \theta((i \ \ i+1)) = \rho_i^{0}(=\rho_i) \ \ \text{for } \  i=1,2,\dots, n-1. \] Then $\theta$ is a splitting homomorphism of  $\Phi_{n,k}$, and hence we get that $$M_kVT_n \cong M_kVPT_n \rtimes S_n.$$
To determine a generating set and defining relations for $M_kVPT_n$, we apply the Reidemeister Schreier method. As a Schreier set of coset representatives for $M_kVPT_n$ in $M_kVT_n$, we choose the same set $\bigwedge_n$ that was used in the computation of generators and relations for the virtual twin group $VT_n$ in~\cite{TNM}. 
\[ \bigwedge_n =\left\{ \prod_{k=2}^{n} m_{k,j_k}\  |\  1\leq j_k \leq k \right\}, \] where $m_{k,l} = \rho_{k-1}\rho_{k-2}\dots \rho_l \text{ for } l < k $ and $m_{k,l}=1$ for other cases.
We define the following elements, expressed in terms of the generators of $M_kVT_n$, by the formulas:
\begin{align*}
    \lambda_{i,i+1}^0 &= \rho_is_i,\\
    \lambda_{i+1,i}^0 &=\rho_i \lambda_{i,i+1}^0 \rho_i = s_i\rho_i,\\
    \lambda_{i,i+1}^{\beta} &= \rho_i\rho_i^{\beta}
\end{align*}
for $i=1,2,\dots,n-1$ and $\beta=1,2,\ldots, k-1$, and 
\begin{align*}
    \lambda_{i,j}^0 &= \rho_{j-1} \rho_{j-2} \dots \rho_{i+1} \lambda_{i,i+1}^0 \rho_{i+1}\dots \rho_{j-2} \rho_{j-1},\\
     \lambda_{j,i}^0 &= \rho_{j-1} \rho_{j-2} \dots \rho_{i+1} \lambda_{i+1,i}^0 \rho_{i+1}\dots \rho_{j-2} \rho_{j-1},\\
      \lambda_{i,j}^{\beta} &= \rho_{j-1} \rho_{j-2} \dots \rho_{i+1} \lambda_{i,i+1}^{\beta} \rho_{i+1}\dots \rho_{j-2} \rho_{j-1}
\end{align*} 
for $1\leq i<j-1 \leq n-1$ and $\beta=1,2,\ldots, k-1.$

\begin{lemma}
   Let $a$ be an element of $\big \langle \rho_1,\rho_2,\dots, \rho_{n-1} \rangle $ and $\Bar{a}$ be its image in $S_n$ under the isomorphism $\rho_i \mapsto (i \ \ i+1), ~~i=1,2,\dots,n-1.$ Then for any generator $\lambda_{i,j}^{\alpha}$ of $M_kVPT_n$ the following holds: \[ a^{-1} \lambda_{i,j}^{\alpha} a = \lambda_{(i)\Bar{a},(j)\Bar{a}}^{\alpha} , \] where $(l)\Bar{a}$ is the image of $l$ under the action of the permutation $\Bar{a}$.
\end{lemma}

\begin{proof}
    We prove the lemma by considering $a$ as a generator of the group $$\big \langle \rho_1,\rho_2,\dots, \rho_{n-1} \rangle .$$ Specifically, we take $a=\rho_k$ for various values of $k$. We have:
    \[
     \lambda_{i,j}^{\beta} = \rho_{j-1} \rho_{j-2} \dots \rho_{i+1} \lambda_{i,i+1}^{\beta} \rho_{i+1}\dots \rho_{j-2} \rho_{j-1}.
    \]
Remark that it suffices to consider the case \(i<j\), as the other cases follow similarly.
\begin{itemize}
        \item[(a)] Assume that $k<i-1$ or $k>j.$ Then we have
        \[
        \rho_k \lambda_{i,j}^{\beta} \rho_k = \rho_k (\rho_{j-1} \rho_{j-2} \dots \rho_{i+1} \lambda_{i,i+1}^{\beta} \rho_{i+1}\dots \rho_{j-2} \rho_{j-1})\rho_k.
        \]
        In this case, $\rho_k$ is permutable with $\rho_i,\rho_{i+1}, \dots,\rho_{j-1}$ , $s_i$, and $\rho_i^{\beta}$. Hence 
        \[
        \rho_k \lambda_{i,j}^{\beta} \rho_k = \lambda_{i,j}^{\beta}.
        \]
        Similarly, we obtain that 
        $$\rho_k \lambda_{j,i}^{\beta} \rho_k = \lambda_{j,i}^{\beta}.$$ 
        \item[(b)] Assume that $i<k<j-1$. Then we have
        \begin{align*}
            \rho_k \lambda_{i,j}^{\beta} \rho_k &= \rho_k (\rho_{j-1} \rho_{j-2} \dots \rho_{i+1} \lambda_{i,i+1}^{\beta} \rho_{i+1}\dots \rho_{j-2} \rho_{j-1})\rho_k\\
            &=\rho_{j-1} \dots  \rho_{k+2}(\rho_{k}\rho_{k+1}\rho_{k})\dots\rho_{i+1} \lambda_{i,j}^{\beta}\rho_{i+1}\dots (\rho_{k}\rho_{k+1}\rho_{k}) \rho_{k+2} \dots \rho_{j-1}.
        \end{align*}
        Using the relation (3.6), we have
        \begin{align*}
        & \rho_k \lambda_{i,j}^{\beta} \rho_k \\
            &= \rho_{j-1} \dots  \rho_{k+2}\rho_{k+1}\rho_{k} (\rho_{k+1}\rho_{k-1} \dots \rho_{i+1}\lambda_{i,j}^{\beta}\rho_{i+1}\dots \rho_{k-1}\rho_{k+1})\rho_{k}\rho_{k+1} \rho_{k+2} \dots \rho_{j-1}\\
            &= \rho_{j-1} \dots \rho_{k}(\rho_{k+1} \lambda_{i,k}^{\beta}\rho_{k+1})\rho_{k} \dots \rho_{j-1}.
        \end{align*}
       By case (a) we have
        \[
        \rho_{k+1} \lambda_{i,k}^{\beta}\rho_{k+1} = \lambda_{ik},
        \]
        and so we finally get that
        \[
        \rho_{j-1} \dots \rho_{k}(\rho_{k+1} \lambda_{i,k}^{\beta}\rho_{k+1})\rho_{k} \dots \rho_{j-1}= \lambda_{ij}^{\beta}.
        \]
        Similarly, we obtain that 
        $$\rho_k \lambda_{j,i}^{\beta} \rho_k = \lambda_{j,i}^{\beta}.$$
        \item[(c)] Assume now that $k=i-1$. Then we have
\begin{align*}
     \rho_k \lambda_{i,j}^{\beta} \rho_k &= \rho_{i-1} (\rho_{j-1} \rho_{j-2} \dots \rho_{i+1} \lambda_{i,i+1}^{\beta} \rho_{i+1}\dots \rho_{j-2} \rho_{j-1})\rho_{i-1}\\
     &= \rho_{j-1} \rho_{j-2} \dots \rho_{i+1} (\rho_{i-1} \lambda_{i,i+1}^{\beta} \rho_{i-1}) \rho_{i+1}\dots \rho_{j-2} \rho_{j-1}. ~~~~-(*)
\end{align*}
    Now using the equalities 
    \begin{align*}
        \lambda_{i,i+1}^0 &= \rho_is_i,\\
    \lambda_{i,i+1}^{\beta} &= \rho_i\rho_i^{\beta}, 1\leq \beta \leq k-1,
    \end{align*}
    we get that
    \begin{align*}
        \rho_{i-1} \lambda_{i,i+1}^{\beta} \rho_{i-1} &= \rho_{i-1} \rho_i s_i \rho_{i-1} \text{ if } \beta=0,\\
         \rho_{i-1} \lambda_{i,i+1}^{\beta} \rho_{i-1} &=\rho_{i-1} \rho_i \rho_i^{\beta} \rho_{i-1} \text{ if } \beta \neq 0.
    \end{align*}
    In both the cases, after using the relations of $M_kVT_n$, we obtain that
    \[
    \hspace{2.5cm} \rho_{i-1} \lambda_{i,i+1}^{\beta} \rho_{i-1} = \rho_{i} \lambda_{i-1,i}^{\beta} \rho_{i}. \hspace{3cm}~-(**)
    \]
    Using $(*)$ and $(**)$, we clearly obtain that 
    $$\rho_k \lambda_{i,j}^{\beta} \rho_k = \lambda_{i-1,j}^{\beta}.$$
Again, in a similar manner, we obtain that 
$$\rho_k \lambda_{j,i}^{\beta} \rho_k = \lambda_{j,i-1}^{\beta}.$$
\item[(d)] Assume that $k=i$ and $i<j-1$. Similar calculations to the cases (a)--(c) imply that
\begin{align*}
    \rho_i \lambda_{i,i+1}^{\beta} \rho_i &= \lambda_{i+1,i}^{\beta},\\
    \rho_i \lambda_{i,j}^{\beta} \rho_i &= \lambda_{i+1,j}^{\beta},\\
\rho_i \lambda_{i+1,i}^{\beta} \rho_i &= \lambda_{i,i+1}^{\beta},\\
\rho_i \lambda_{j,i}^{\beta} \rho_i &= \lambda_{j,i+1}^{\beta}.
\end{align*}
\item[(e)] Assume now that $i+1<j$. Again, similar calculations done in the cases (a)--(c) give that
\begin{align*}
    \rho_{j-1} \lambda_{i,j}^{\beta} \rho_{j-1} &= \lambda_{i,j-1}^{\beta},\\
    \rho_{j-1} \lambda_{j,i}^{\beta} \rho_{j-1} &= \lambda_{j-1,i}^{\beta}.
\end{align*}
\item[(f)] Finally, assume that $k=j$. Also, similar calculations done in the cases (a)--(c) imply that
\begin{align*}
    \rho_{j} \lambda_{i,j}^{\beta} \rho_{j} &= \lambda_{i,j+1}^{\beta},\\
    \rho_{j} \lambda_{j,i}^{\beta} \rho_{j} &= \lambda_{j+1,i}^{\beta}.
\end{align*}
        \end{itemize}
\end{proof}

\begin{theorem}
    The group $M_kVPT_n$ admits a presentation with the generators \[\lambda_{i,j}^0,~ 1\leq i \neq j \leq n, \text{ and } \lambda_{i,j}^{\beta},~ 1\leq i < j \leq n, \  \beta=1,2,\ldots, k-1,\]
    and the following defining relations: 
    \begin{align}
        \lambda_{i,j}^{0} \lambda_{k,l}^{0} &= \lambda_{k,l}^{0} \lambda_{i,j}^{0},\\
        \lambda_{i,j}^{\alpha} \lambda_{k,l}^{\gamma} &= \lambda_{k,l}^{\gamma} \lambda_{i,j}^{\alpha},~1 \leq \alpha \leq \gamma \leq k-1,\\
        \lambda_{i,j}^{\alpha} \lambda_{i,k}^{\alpha} \lambda_{j,k}^{\alpha} &= \lambda_{j,k}^{\alpha} \lambda_{i,k}^{\alpha} \lambda_{i,j}^{\alpha},~ \alpha=1,2,\ldots ,k-1,\\ 
         \lambda_{i,j}^{\alpha} \lambda_{i,k}^{\beta} \lambda_{j,k}^{\beta} &= \lambda_{j,k}^{\beta} \lambda_{i,k}^{\beta} \lambda_{i,j}^{\alpha},~1 \leq \alpha < \beta \leq k-1,\\
          \lambda_{i,j}^{\alpha} \lambda_{i,k}^{\alpha} \lambda_{j,k}^{\beta} &= \lambda_{j,k}^{\beta} \lambda_{i,k}^{\alpha} \lambda_{i,j}^{\alpha},~ 1\leq  \alpha < \beta \leq k-1,\\
           \lambda_{i,j}^{\alpha} \lambda_{i,k}^{\alpha} \lambda_{j,k}^{0} &= \lambda_{j,k}^{0} \lambda_{i,k}^{\alpha} \lambda_{i,j}^{\alpha},  \ \alpha=1,2,\ldots ,k-1,\\
          \lambda_{i,k}^{\beta} \lambda_{j,k}^{\beta} &= \lambda_{j,k}^{\beta} \lambda_{i,k}^{\beta}, \ \beta=1,2,\ldots ,k-1,
          \end{align}
    where distinct letters $i,j,k,l$ stand for distinct indices.
\end{theorem}
\begin{proof}
    Let $\Bar{}: M_kVT_n \longrightarrow \bigwedge_n $ be a map that assigns to each element $x \in M_kVT_n$ its chosen coset representative $\Bar{x} \in \bigwedge_n.$ By construction, the element $x \Bar{x}^{-1}$ lies in the subgroup $M_kVPT_n$. According to the Reidemeister Schreier theorem \cite[Theorem~2.7]{MKS}, the subgroup $M_kVPT_n$ is generated by elements of the form 
    \[ s_{\lambda,a} = \lambda a \overline{(\lambda a)}^{-1},\] where $\lambda \in \bigwedge_n,  a \in \{s_1,s_2,\dots,s_{n-1},\rho_1^{\alpha},\rho_2^{\alpha},\dots,\rho_{n-1}^{\alpha}\}, \text{ and } \alpha =0,1,\dots,k-1. $
    \begin{itemize}
        \item [Case 1:]  If $a=s_i$, then using Lemma 3.1, we get:
        \[s_{\lambda,s_i}= (\lambda s_i)\overline{(\lambda s_i)}^{-1}=\lambda s_i \rho_i \lambda^{-1}= \lambda \lambda_{i,i+1}^{-1} \lambda^{-1} = \lambda_{j,l} \text{ for } 1\leq j \neq l \leq n.\]
        \item [Case 2:] If $a=\rho_i^{\alpha}$, we consider the following sub cases:
        \begin{itemize}
            \item[(a)] If $\alpha=0$, then  
            \[s_{\lambda,\rho_i^0}=(\lambda \rho_i^0)\overline{(\lambda \rho_i^0)}^{-1}=e \text{ for } 1\leq i \leq n-1.\]
            \item[(b)] If $\alpha \neq 0$, then
            \[s_{\lambda,\rho_i^{\alpha}}=(\lambda \rho_i^{\alpha})\overline{(\lambda \rho_i^{\alpha})}^{-1}=\lambda \rho_i^{\alpha} \rho_i^{0} \lambda^{-1} = (\lambda \lambda_{i,i+1}^{\alpha} \lambda^{-1})^{-1} =\lambda_{j,l}^{\alpha}\] for $1\leq j < l \leq n$ and $1\leq \alpha \leq k-1.$
        \end{itemize}
    \end{itemize}
    Therefore, the group $M_kVPT_n$ is generated by the elements $\lambda_{k,l}^0$ for $1 \leq k \neq l \leq n$ together with $\lambda_{i,j}^{\beta}$ for $1 \leq i < j \leq n$ and $1 \leq \beta \leq k-1.$
    
    Now, to determine the defining relations of $M_kVPT_n$, we introduce a rewriting process $\tau$. This process assigns to each word $u$ written in the generators of $M_kVT_n$, which represents an element of $M_kVPT_n$, a word $\tau(u)$ expressed in the generators of $M_kVPT_n$.
    Let \[u=x_1^{\epsilon_1}x_2^{\epsilon_2} \dots x_r^{\epsilon_r}, ~\epsilon_l=\pm1\]
    be a reduced word in the generators  
    \[ x_l \in \{s_1,s_2,\dots,s_{n-1},\rho_1^0,\rho_2^0,\dots,\rho_{n-1}^0, \dots, \rho_1^{k-1}, \rho_2^{k-1},\dots,\rho_{n-1}^{k-1} \}. \]
    We associate to $u$ the word \[ \tau(u)= s_{k_1,x_1}^{\epsilon_1} s_{k_2,x_2}^{\epsilon_2} \dots s_{k_r,x_r}^{\epsilon_r},\]
    which is written in the generators of $M_kVPT_n$.
    Here, for each $j=1,2,\dots,r,$ the element $k_j \in \bigwedge_n$ is defined as follows:
    \begin{itemize}
        \item if $\epsilon_j=1$, then $k_j$ is the Schreier representative of the $(j-1)^{th}$ initial segment of the word $u$,
        \item if $\epsilon_j=-1$, then $k_j$ is the Schreier representative of the $(j)^{th}$ initial segment of the word $u$.
    \end{itemize}
    By \cite[Theorem~2.9]{MKS}, the group $M_kVPT_n$ is defined by the relations   \[ r_{\mu,\lambda} = \tau(\lambda r_{\mu} \lambda^{-1}) = \lambda \tau(r_{\mu}) \lambda^{-1}, \lambda \in \bigwedge_n,  \] where $r_{\mu}$ is a defining relation of $M_kVT_n.$
    We now derive the defining relations of $M_kVPT_n$ corresponding to the relations of $M_kVT_n$ in the following cases:
    \begin{itemize}
        \item[i)] For the relation (3.1), let $r_1 = s_i^{2} =1$, $i=1,2,\ldots,n-1$. Then, we have
    \begin{align*}
             r_{1,e}&=\tau(r_1) \\
             &= s_{e,s_i}s_{\rho_i,s_i}\\
             &= (s_i \rho_i)(\rho_i s_i \rho_i \rho_i)\\
             &= e.
        \end{align*}
         which do not give rise to any nontrivial relations in $M_kVPT_n$.
        \item[ii)] For the relations (3.2), let $r_2= s_i s_j s_i^{-1} s_j^{-1}=s_i s_j s_i s_j=1,~|i-j|\geq2$. Then, we have
        \begin{align*}
            r_{2,e} &= \tau(r_2)\\ &= s_{e,s_i}s_{\rho_i,s_j}s_{\rho_i \rho_j,s_i}s_{\rho_i \rho_j \rho_i,s_j}\\
            &=(s_i \rho_i)(s_j \rho_j)(\rho_i s_i)(\rho_j s_j)\\
            &=(\lambda_{i,i+1}^0)^{-1} (\lambda_{j,j+1}^0)^{-1} (\lambda_{i,i+1}^0) (\lambda_{j,j+1}^0).
        \end{align*}
        Thus, using Lemma (3.1), and for $\lambda \in \bigwedge_n$, we get that
        \begin{align*}
            r_{2,\lambda} &= \tau(\lambda r_{2} \lambda^{-1}) \\&= \lambda \tau(r_{2}) \lambda^{-1}\\
            &= \lambda (\lambda_{i,i+1}^0)^{-1}(\lambda_{j,j+1}^0)^{-1} (\lambda_{i,i+1}^0) (\lambda_{j,j+1}^0) \lambda^{-1}\\
            &= (\lambda_{i,j}^0)^{-1} (\lambda_{k,l}^0)^{-1} (\lambda_{i,j}^0) (\lambda_{k,l}^0),
        \end{align*}
        which gives the relation (3.10).

        \item[iii)] For the relation (3.4), let $r_3=\rho_i^{\alpha} \rho_j^{\beta} \rho_i^{\alpha} \rho_j^{\beta},~ \alpha \neq 0,\beta\neq 0, |i-j| \geq 2$. Then, we have
        \begin{align*}
            r_{3,e} &= \tau(r_3)\\
            &=s_{e,\rho_i^{\alpha}} s_{\rho_i, \rho_j^{\beta}} s_{\rho_i\rho_j,\rho_i^{\alpha}}s_{\rho_i\rho_j\rho_i,\rho_j^{\beta}}\\
            &=(\lambda_{i,i+1}^{\alpha})^{-1} (\lambda_{j,j+1}^{\beta})^{-1} (\lambda_{i,i+1}^{\alpha}) (\lambda_{j,j+1}^{\beta}).
        \end{align*}
        Thus, using Lemma (3.1), and for $\lambda \in \bigwedge_n$, we get that
        \begin{align*}
             r_{3,\lambda} &= \tau(\lambda r_{3} \lambda^{-1})\\ & = \lambda \tau(r_{3}) \lambda^{-1}\\
             &= (\lambda_{i,j}^{\alpha})^{-1} (\lambda_{k,l}^{\beta})^{-1} (\lambda_{i,j}^{\alpha}) (\lambda_{k,l}^{\beta}),
        \end{align*}
        which gives the relation (3.11). 
        \item[iv)] For the relations (3.6), let $r_4=\rho_{i}^{\alpha}\rho_{i+1}^{\alpha}\rho_{i}^{\alpha}(\rho_{i+1}^{\alpha})^{-1} (\rho_{i}^{\alpha})^{-1}(\rho_{i+1}^{\alpha})^{-1},\alpha \neq 0$. Then, we have
        \begin{align*}
            r_{4,e} &= \tau(r_4)\\ &=s_{e,\rho_i^{\alpha}}s_{\rho_i,\rho_{i+1}^{\alpha}}s_{\rho_i\rho_{i+1},\rho_{i}^{\alpha}}s_{\rho_i\rho_{i+1}\rho_i,\rho_{i+1}^{\alpha}}s_{\rho_i\rho_{i+1}\rho_i \rho_{i+1},\rho_{i}^{\alpha}} s_{\rho_i\rho_{i+1}\rho_i \rho_{i+1}\rho_i,\rho_{i+1}^{\alpha}}\\
            &=(\lambda_{i,i+1}^{\alpha})^{-1} (\lambda_{i,i+2}^{\alpha})^{-1} (\lambda_{i+1,i+2}^{\alpha})^{-1} (\lambda_{i,i+1}^{\alpha}) (\lambda_{i,i+2}^{\alpha}) (\lambda_{i+1,i+2}^{\alpha}).
        \end{align*}
        Again, using Lemma (3.1), and for $\lambda \in \bigwedge_n$, we get that
       \begin{align*}
            r_{4,\lambda} &= \tau(\lambda r_{4} \lambda^{-1})\\ &= \lambda \tau(r_{4}) \lambda^{-1}\\
            &= (\lambda_{j,k}^{\alpha})^{-1} (\lambda_{i,k}^{\alpha})^{-1} (\lambda_{i,j}^{\alpha})^{-1} (\lambda_{j,k}^{\alpha})(\lambda_{i,k}^{\alpha})(\lambda_{i,j}^{\alpha}),
       \end{align*}
       which gives the relation (3.12). On the other hand, for $\alpha =0$ in this case we get trivial relation.
    \end{itemize}
   The remaining relations are derived in a similar manner as above. In particular, relations (3.13), (3.14), and (3.15) follow from relations (3.7), (3.8), and (3.9), respectively, using arguments analogous to those in case iv). Moreover, relation (3.16) is obtained from relation (3.7) by taking $\alpha = 0$. The remaining relations of $M_kVT_n$ yield only trivial relations in $M_kVPT_n$ and this completes the proof.
   \end{proof}

\subsection{Multi-Virtual Semi-Pure Twin Group}

We now introduce another subgroup of the multi-virtual twin group, called the multi-virtual semi-pure twin group, denoted by \(M_kVHT_n\). Define the map $$\psi_{n,k}: M_kVT_n \longrightarrow S_n $$ given by \[ \Psi_{n,k}(s_i)=e \ \text{ and }\  \Psi_{n,k}(\rho_i^{\alpha})=(i \ \ i+1) \]
for $i=1,2,\dots,n-1$ and $\alpha = 0,1,\dots,k-1,$ where $e$ is the identity element of $S_n$. We define the multi-virtual semi-pure twin group \(M_kVHT_n\) as the kernel of the map $\Psi_{n,k}$. Then, $M_kVHT_n$ is a normal subgroup of $M_kVT_n$, and we obtain the following short exact sequence
\[
1 \longrightarrow M_kVHT_n \longrightarrow M_kVT_n \longrightarrow S_n \longrightarrow 1.
\] 
Hence, we can see that $$M_kVT_n \cong M_kVHT_n \rtimes S_n$$ in a similar manner as done in the previous subsection for $M_kVPT_n$.

Now, to determine a generating set and defining relations for $M_kVHT_n$, we apply the Reidemeister Schreier method as we did in the previous subsection. We define the following elements, expressed in terms of the generators of $M_kVT_n$, by the formulas:
\begin{align*}
    \kappa_{i,i+1}^0 &= s_i,\\
    \kappa_{i+1,1}^0 &=\rho_i \kappa_{i,i+1}^0 \rho_i = \rho_is_i\rho_i,\\
    \kappa_{i,i+1}^{\beta} &= \rho_i\rho_i^{\beta}, 1\leq \beta \leq k-1
\end{align*}
for $i=1,2,\dots,n-1$, and 
\begin{align*}
    \kappa_{j,i}^0 &= \rho_{j-1} \rho_{j-2} \dots \rho_{i+1} \kappa_{i+1,i}^0 \rho_{i+1}\dots \rho_{j-2} \rho_{j-1},\\
      \kappa_{i,j}^{\beta} &= \rho_{j-1} \rho_{j-2} \dots \rho_{i+1} \kappa_{i,i+1}^{\beta} \rho_{i+1}\dots \rho_{j-2} \rho_{j-1}
\end{align*} 
for $1\leq i<j-1 \leq n-1$ and  $\beta=1,2,\ldots, k-1.$

\begin{theorem}
    The group $M_kVHT_n$ admits a presentation with the generators \[\kappa_{i,j}^0,~ 1\leq i \neq j \leq n, \text{ and } \kappa_{i,j}^{\beta},~ 1\leq i < j \leq n, ~1\leq \beta \leq k-1,\]
    and the following defining relations:
    \begin{align}
        \kappa_{i,j}^{\alpha} \kappa_{k,l}^{\gamma} &= \kappa_{k,l}^{\gamma} \kappa_{i,j}^{\alpha},~1 \leq \alpha \leq \gamma \leq k-1,\\
        \kappa_{i,j}^{\alpha} \kappa_{i,k}^{\alpha} \kappa_{j,k}^{\alpha} &= \kappa_{j,k}^{\alpha} \kappa_{i,k}^{\alpha} \kappa_{i,j}^{\alpha},~  
       \alpha=1,2,\ldots k-1,\\
         \kappa_{i,j}^{\alpha} \kappa_{i,k}^{\beta} \kappa_{j,k}^{\beta} &= \kappa_{j,k}^{\beta} \kappa_{i,k}^{\beta} \kappa_{i,j}^{\alpha},~ 
         1 \leq \alpha < \beta \leq k-1,\\
          \kappa_{i,j}^{\alpha} \kappa_{i,k}^{\alpha} \kappa_{j,k}^{\beta} &= \kappa_{j,k}^{\beta} \kappa_{i,k}^{\alpha} \kappa_{i,j}^{\alpha},~ 
          1\leq  \alpha < \beta \leq k-1,\\
           \kappa_{i,k}^{\alpha} \kappa_{i,j}^{\alpha} \kappa_{j,k}^{0} &= \kappa_{j,k}^{0} \kappa_{i,k}^{\alpha} \kappa_{i,j}^{\alpha},~  
       \alpha=1,2,\ldots k-1,\\
          \kappa_{i,k}^{\beta} \kappa_{j,k}^{\beta} &= \kappa_{j,k}^{\beta} \kappa_{i,k}^{\beta}, ~  
       \beta=1,2,\ldots k-1,
    \end{align}
    where distinct letters $i,j,k,l$ stand for distinct indices.
\end{theorem}
\begin{proof}
The proof is similar to that of Theorem 3.2 and is therefore omitted.
\end{proof}
 
\section{Complex Homogeneous Local Representations of $M_kVT_n$}
In this section, our aim is to construct and study matrix representations of the multi-virtual twin group \(M_kVT_n\). We classify and study all complex homogeneous \(2\)-local representations of \(M_kVT_n\) for \(k \geq 1\) and \(n \geq 3\).
\begin{theorem} \label{thm 41}
    Let $\zeta: M_kVT_n \longrightarrow \mathrm{GL}_{n}(\mathbb{C})$ be a complex homogeneous 2-local representation of $M_kVT_n$ for  \(k \geq 1\) and \(n \geq 3\). Then, $\zeta$ is equivalent to one of the following eight representations.
    \begin{itemize}
    \item[(1)] $\zeta_1: M_kVT_n \longrightarrow \mathrm{GL}_n(\mathbb{C})$ such that
$$\zeta_1(s_i)=\zeta_1(\rho_i^{\alpha})= I_n$$
    for $1\leq i \leq n-1, 0 \leq \alpha \leq k-1.$ 
\item[(2)] $\zeta_2: M_kVT_n \longrightarrow \mathrm{GL}_n(\mathbb{C})$ such that
{\small
$$\zeta_2(s_i)=\left( \begin{array}{c|@{}c|c@{}}
   \begin{matrix}
     I_{i-1} 
   \end{matrix} 
      & \textbf{0} & \textbf{0} \\
      \hline
    \textbf{0} &\hspace{0.2cm} \begin{matrix}
   		1 & 0\\
   		0 & 1\\
   		\end{matrix}  & \textbf{0}  \\
\hline
\textbf{0} & \textbf{0} & I_{n-i-1}
\end{array} \right) \text{ and } \zeta_2(\rho_i^{\alpha})=\left( \begin{array}{c|@{}c|c@{}}
   \begin{matrix}
     I_{i-1} 
   \end{matrix} 
      & \textbf{0} & \textbf{0} \\
      \hline
    \textbf{0} &\hspace{0.2cm} \begin{matrix}
   		0 & \dfrac{1}{y_{\alpha}}\\
   		y_{\alpha} & 0\\
   		\end{matrix}  & \textbf{0}  \\
\hline
\textbf{0} & \textbf{0} & I_{n-i-1}
\end{array} \right)$$
}
for $1\leq i \leq n-1, 0 \leq \alpha \leq k-1, \text{ where } y_{\alpha}\in \mathbb{C}^*.$
\item[(3)] $\zeta_3: M_kVT_n \longrightarrow \mathrm{GL}_n(\mathbb{C})$ such that
{\small
$$\zeta_3(s_i)=\left( \begin{array}{c|@{}c|c@{}}
   \begin{matrix}
     I_{i-1} 
   \end{matrix} 
      & \textbf{0} & \textbf{0} \\
      \hline
    \textbf{0} &\hspace{0.2cm} \begin{matrix}
   		1 & 0\\
   		0 & -1\\
   		\end{matrix}  & \textbf{0}  \\
\hline
\textbf{0} & \textbf{0} & I_{n-i-1}
\end{array} \right) \text{ and } \zeta_3(\rho_i^{\alpha})=\left( \begin{array}{c|@{}c|c@{}}
   \begin{matrix}
     I_{i-1} 
   \end{matrix} 
      & \textbf{0} & \textbf{0} \\
      \hline
    \textbf{0} &\hspace{0.2cm} \begin{matrix}
   		0 & \dfrac{1}{y_{\alpha}}\\
   		y_{\alpha} & 0\\
   		\end{matrix}  & \textbf{0}  \\
\hline
\textbf{0} & \textbf{0} & I_{n-i-1}
\end{array} \right)$$
}
$\text{ for } 1\leq i \leq n-1, 0 \leq \alpha \leq k-1, \text{ where } y_{\alpha}\in \mathbb{C}^*.$
\item[(4)] $\zeta_4: M_kVT_n \longrightarrow \mathrm{GL}_n(\mathbb{C})$ such that
{\small
$$\zeta_4(s_i)=\left( \begin{array}{c|@{}c|c@{}}
   \begin{matrix}
     I_{i-1} 
   \end{matrix} 
      & \textbf{0} & \textbf{0} \\
      \hline
    \textbf{0} &\hspace{0.2cm} \begin{matrix}
   		-1 & 0\\
   		0 & 1\\
   		\end{matrix}  & \textbf{0}  \\
\hline
\textbf{0} & \textbf{0} & I_{n-i-1}
\end{array} \right) \text{ and } \zeta_4(\rho_i^{\alpha})=\left( \begin{array}{c|@{}c|c@{}}
   \begin{matrix}
     I_{i-1} 
   \end{matrix} 
      & \textbf{0} & \textbf{0} \\
      \hline
    \textbf{0} &\hspace{0.2cm} \begin{matrix}
   		0 & \dfrac{1}{y_{\alpha}}\\
   		y_{\alpha} & 0\\
   		\end{matrix}  & \textbf{0}  \\
\hline
\textbf{0} & \textbf{0} & I_{n-i-1}
\end{array} \right)$$
}
$\text{ for } 1\leq i \leq n-1, 0 \leq \alpha \leq k-1, \text{ where } y_{\alpha}\in \mathbb{C}^*.$
\item[(5)] $\zeta_5: M_kVT_n \longrightarrow \mathrm{GL}_n(\mathbb{C})$ such that
{\small
$$\zeta_5(s_i)=\left( \begin{array}{c|@{}c|c@{}}
   \begin{matrix}
     I_{i-1} 
   \end{matrix} 
      & \textbf{0} & \textbf{0} \\
      \hline
    \textbf{0} &\hspace{0.2cm} \begin{matrix}
   		-1 & 0\\
   		0 & -1\\
   		\end{matrix}  & \textbf{0}  \\
\hline
\textbf{0} & \textbf{0} & I_{n-i-1}
\end{array} \right) \text{ and } \zeta_5(\rho_i^{\alpha})=\left( \begin{array}{c|@{}c|c@{}}
   \begin{matrix}
     I_{i-1} 
   \end{matrix} 
      & \textbf{0} & \textbf{0} \\
      \hline
    \textbf{0} &\hspace{0.2cm} \begin{matrix}
   		0 & \dfrac{1}{y_{\alpha}}\\
   		y_{\alpha} & 0\\
   		\end{matrix}  & \textbf{0}  \\
\hline
\textbf{0} & \textbf{0} & I_{n-i-1}
\end{array} \right)$$
}
$\text{ for } 1\leq i \leq n-1, 0 \leq \alpha \leq k-1, \text{ where } y_{\alpha}\in \mathbb{C}^*.$
\item[(6)] $\zeta_6: M_kVT_n \longrightarrow \mathrm{GL}_n(\mathbb{C})$ such that
{\small
$$\zeta_6(s_i)=\left( \begin{array}{c|@{}c|c@{}}
   \begin{matrix}
     I_{i-1} 
   \end{matrix} 
      & \textbf{0} & \textbf{0} \\
      \hline
    \textbf{0} &\hspace{0.2cm} \begin{matrix}
   		1 & z\\
   		0 & -1\\
   		\end{matrix}  & \textbf{0}  \\
\hline
\textbf{0} & \textbf{0} & I_{n-i-1}
\end{array} \right) \text{ and } \zeta_6(\rho_i^{\alpha})=\left( \begin{array}{c|@{}c|c@{}}
   \begin{matrix}
     I_{i-1} 
   \end{matrix} 
      & \textbf{0} & \textbf{0} \\
      \hline
    \textbf{0} &\hspace{0.2cm} \begin{matrix}
   		0 & \dfrac{1}{y_{\alpha}}\\
   		y_{\alpha} & 0\\
   		\end{matrix}  & \textbf{0}  \\
\hline
\textbf{0} & \textbf{0} & I_{n-i-1}
\end{array} \right)$$
}
$\text{ for } 1\leq i \leq n-1, 0 \leq \alpha \leq k-1, \text{ where } y_{\alpha}, z \in \mathbb{C}^*.$
\item[(7)] $\zeta_7: M_kVT_n \longrightarrow \mathrm{GL}_n(\mathbb{C})$ such that
{\small
$$\zeta_7(s_i)=\left( \begin{array}{c|@{}c|c@{}}
   \begin{matrix}
     I_{i-1} 
   \end{matrix} 
      & \textbf{0} & \textbf{0} \\
      \hline
    \textbf{0} &\hspace{0.2cm} \begin{matrix}
   		-1 & z\\
   		0 & 1\\
   		\end{matrix}  & \textbf{0}  \\
\hline
\textbf{0} & \textbf{0} & I_{n-i-1}
\end{array} \right) \text{ and } \zeta_7(\rho_i^{\alpha})=\left( \begin{array}{c|@{}c|c@{}}
   \begin{matrix}
     I_{i-1} 
   \end{matrix} 
      & \textbf{0} & \textbf{0} \\
      \hline
    \textbf{0} &\hspace{0.2cm} \begin{matrix}
   		0 & \dfrac{1}{y_{\alpha}}\\
   		y_{\alpha} & 0\\
   		\end{matrix}  & \textbf{0}  \\
\hline
\textbf{0} & \textbf{0} & I_{n-i-1}
\end{array} \right)$$
}
$\text{ for } 1\leq i \leq n-1, 0 \leq \alpha \leq k-1, \text{ where } y_{\alpha}, z \in \mathbb{C}^*.$
\item[(8)] $\zeta_8: M_kVT_n \longrightarrow \mathrm{GL}_{n}(\mathbb{C})$ such that
{\small
\begin{align*}
    \zeta_8(s_i) &= \left( \begin{array}{c|@{}c|c@{}}
   \begin{matrix}
     I_{i-1} 
   \end{matrix} 
      & \textbf{0} & \textbf{0} \\
      \hline
    \textbf{0} &\hspace{0.2cm} \begin{matrix}
   		-a & -\dfrac{(a^2-1)}{b}\\
   		b & a\\
   		\end{matrix}  & \textbf{0}  \\
\hline
\textbf{0} & \textbf{0} & I_{n-i-1}
\end{array} \right) \text{ and },\\
\zeta_8(\rho_i^{\alpha}) &= \left( \begin{array}{c|@{}c|c@{}}
   \begin{matrix}
     I_{i-1} 
   \end{matrix} 
      & \textbf{0} & \textbf{0} \\
      \hline
    \textbf{0} &\hspace{0.2cm} \begin{matrix}
   		0 & \dfrac{1}{y_{\alpha}}\\
   		y_{\alpha} & 0\\
   		\end{matrix}  & \textbf{0}  \\
\hline
\textbf{0} & \textbf{0} & I_{n-i-1}
\end{array} \right)
\end{align*}
}
$\text{ for } 1\leq i \leq n-1, 0 \leq \alpha \leq k-1, \text{ where } a\in \mathbb{C}$ and $b, y_{\alpha}\in \mathbb{C}^*.$
\end{itemize}

\end{theorem}
\begin{proof}
    Let $\zeta$ be a complex homogeneous $2$-local  representation of $M_kVT_n$ for  \(k \geq 1\) and \(n \geq 3\). As $M_kVT_n$ has $k+1$ families of generators, which are $s_1,s_2,\dots,s_{n-1},\rho_1^{\alpha},\rho_2^{\alpha},\dots,\rho_{n-1}^{\alpha}$ for $\alpha=0,1,\dots,k-1,$ the images of the generators of $M_kVT_n$ can be defined as follows.
    For $ 1\leq i \leq n-1$ and $0 \leq \alpha \leq k-1,$ we set
    {\small
    $$\zeta(s_i)=\left( \begin{array}{c|@{}c|c@{}}
   \begin{matrix}
     I_{i-1} 
   \end{matrix} 
      & \textbf{0} & \textbf{0} \\
      \hline
    \textbf{0} &\hspace{0.2cm} \begin{matrix}
   		a & b\\
   		c & d\\
   		\end{matrix}  & \textbf{0}  \\
\hline
\textbf{0} & \textbf{0} & I_{n-i-1}
\end{array} \right) \text{ and } 
\zeta(\rho_i^{\alpha})=\left( \begin{array}{c|@{}c|c@{}}
   \begin{matrix}
     I_{i-1} 
   \end{matrix} 
      & \textbf{0} & \textbf{0} \\
      \hline
    \textbf{0} &\hspace{0.2cm} \begin{matrix}
   		w_{\alpha} & x_{\alpha}\\
   		y_{\alpha} & z_{\alpha}\\
   		\end{matrix}  & \textbf{0}  \\
\hline
\textbf{0} & \textbf{0} & I_{n-i-1}
\end{array} \right),$$
}
where $a,b,c,d,w_{\alpha},x_{\alpha},y_{\alpha},z_{\alpha} \in \mathbb{C},~ ad-bc \neq 0,w_{\alpha}z_{\alpha}-z_{\alpha}y_{\alpha} \neq 0.$ Remark that we only need to consider the following relations among the generators of $M_kVT_n$ and the remaining relations yield similar equations:
\begin{align*}
    s_1^2 &= 1,\\
    (\rho_1^\alpha )^2 &= 1 \  \text{ for } \alpha=0,1,\dots,k-1,\\
    \rho_1^{\alpha} \rho_{2}^{\alpha} \rho_1^{\alpha} &= \rho_{2}^{\alpha} \rho_1^{\alpha} \rho_{2}^{\alpha} \ \text{ for } \alpha=0,1,\dots,k-1,\\
    \rho_1^{\alpha} \rho_{2}^{\beta} \rho_1^{\beta} &= \rho_{2}^{\beta} \rho_1^{\beta} \rho_{2}^{\alpha} \ \text{ for } 0 \leq \alpha < \beta \leq k-1,   \\
    \rho_1^{\alpha} \rho_{2}^{\alpha} \rho_1^{\beta} &= \rho_{2}^{\beta} \rho_1^{\alpha} \rho_{2}^{\alpha} \  \text{ for } 
    0 \leq \alpha < \beta \leq k-1, \\
     \rho_1^{\alpha} \rho_{2}^{\alpha} s_1 &= s_{2} \rho_1^{\alpha} \rho_{2}^{\alpha} \ \text{ for } \alpha=0,1,\dots,k-1.
\end{align*}

Now, we substitute the images of the generators under the representations $\zeta$ into above relations. This substitution produces a system of $(k+1)^2$ equations involving $4(k+1)$ unknowns. The resulting system of equations are given in the following table:
\begin{tcolorbox}[
  title=System of Equations for {\(\displaystyle \beta = 0,1,2,\dots,k-1, \alpha= 0,1,\dots,k-2 ~\& ~ \alpha < \beta \)},,
  colframe=black,
  colback=white,
  sharp corners,
  fonttitle=\bfseries,
]
\setlength{\columnseprule}{0.4pt}
\vspace{-0.6cm}
\begin{multicols}{2}
{
\begin{align*}
a^2 + bc &= 1,\\
ab + bd &= 0,\\
ac + cd &= 0,\\
d^2 + bc &= 1,
\end{align*}
\vspace{-0.2cm}}
$\begin{array}{c}
\rule{.9\linewidth}{0.4pt}
\end{array}$
{
\begin{align*}
w_{\beta}^2 + x_{\beta} y_{\beta} w_{\beta} &= w_{\beta},\\
w_{\beta} x_{\beta} + w_{\beta} x_{\beta} z_{\beta} &= w_{\beta} x_{\beta},\\
w_{\beta} y_{\beta} + w_{\beta} y_{\beta} z_{\beta} &= w_{\beta} y_{\beta},\\
w_{\beta} z_{\beta}^2 + x_{\beta} y_{\beta} &= z_{\beta} w_{\beta}^2 + x_{\beta} y_{\beta},\\
x_{\beta} z_{\beta} + w_{\beta} x_{\beta} z_{\beta} &= x_{\beta} z_{\beta},\\
y_{\beta} z_{\beta} + w_{\beta} y_{\beta} z_{\beta} &= y_{\beta} z_{\beta},\\
z_{\beta}^2 + x_{\beta} y_{\beta} z_{\beta} &= z_{\beta},
\end{align*}
\vspace{-0.2cm}}
$\begin{array}{c}
\rule{.9\linewidth}{0.4pt}
\end{array}$
{
\begin{align*}
w_{\beta} w_{\alpha} + w_{\beta} y_{\beta} x_{\alpha} &= w_{\beta},\\
x_{\beta} w_{\alpha} + w_{\beta} y_{\beta} x_{\alpha} &= x_{\beta} w_{\alpha},\\
w_{\beta} y_{\alpha} + w_{\beta} y_{\beta} z_{\alpha} &= w_{\beta} y_i,\\
x_{\beta} y_{\alpha} + w_{\beta} z_{\beta} z_{\alpha} &= x_{\beta} y_{\alpha} + w_{\beta} z_{\beta} w_{\alpha},\\
x_{\beta} z_{\alpha} + w_{\beta} z_{\beta} x_{\alpha} &= x_{\beta} z_{\alpha},\\
z_{\beta} y_{\alpha} + y_{\beta} z_{\beta} x_{\alpha} &= y_{\beta} z_{\beta},\\
z_{\beta} z_{\alpha} + y_{\beta} z_{\beta} x_{\alpha} &= z_{\beta},
\end{align*}
\vspace{-0.2cm}}

\columnbreak

{
\begin{align*}
w_{\beta}^2 + x_{\beta} y_{\beta} &= 1,\\
w_{\beta} x_{\beta} + x_{\beta} z_{\beta} &= 0,\\
w_{\beta} y_{\beta} + y_{\beta} z_{\beta} &= 0,\\
z_{\beta}^2 + x_{\beta} y_{\beta} &= 1,
\end{align*}
\vspace{-0.2cm}}
$\begin{array}{c}
\rule{.9\linewidth}{0.4pt}
\end{array}$
{
\begin{align*}
~~w_{\beta} w_{\alpha} + y_{\beta} w_{\alpha} x_{\alpha} &= w_{\alpha},\\
x_{\beta} w_{\alpha} + z_{\beta} w_{\alpha} x_{\alpha} &= w_{\alpha} x_{\alpha},\\
w_{\beta} y_{\alpha} + y_{\beta} w_{\alpha} z_{\alpha} &= w_{\beta} y_{\alpha},\\
x_{\beta} y_{\alpha} + z_{\beta} w_{\alpha} z_{\alpha} &= x_{\beta} y_{\alpha} + w_{\beta} w_{\alpha} z_{\alpha},\\
x_{\beta} z_{\alpha} + w_{\beta} x_{\alpha} z_{\alpha} &= z_{\beta} y_{\alpha},\\
z_{\beta} z_{\alpha} + y_{\beta} x_{\alpha} z_{\alpha} &= z_{\alpha},
\end{align*}
\vspace{-0.2cm}}
$\begin{array}{c}
\rule{.9\linewidth}{0.4pt}
\end{array}$
{
\begin{align*}
a w_{\beta} + c w_{\beta} x_{\beta} &= w_{\beta},\\
b w_{\beta} + d w_{\beta} x_{\beta} &= w_{\beta} x_{\beta},\\
a y_{\beta} + c w_{\beta} z_{\beta} &= a y_{\beta},\\
b y_{\beta} + d w_{\beta} z_{\beta} &= b y_{\beta} + a w_{\beta} z_{\beta},\\
b z_{\beta} + a x_{\beta} z_{\beta} &= x_{\beta} z_{\beta},\\
d y_{\beta} + c w_{\beta} z_{\beta} &= d y_{\beta},\\
d z_{\beta} + c x_{\beta} z_{\beta} &= z_{\beta}.
\end{align*}
\vspace{-0.2cm}}
\end{multicols}
\end{tcolorbox}
\noindent Mathematical techniques of this system show that it admits exactly eight distinct solutions, each of which yields a complex homogeneous \(2\)-local representation of \(M_kVT_n\). Therefore, up to equivalence, there are precisely eight such representations of \(M_kVT_n\), as listed in (1)–(8), and this completes the proof.\end{proof}

In what follows, we investigate the faithfulness of the classified complex homogeneous \(2\)-local representations of \(M_kVT_n\).

\begin{theorem}
Let $
\zeta : M_kVT_n \longrightarrow \mathrm{GL}_n(\mathbb{C})$
be a complex homogeneous \(2\)-local representation of \(M_kVT_n\) with $k \geq 1$ and \(n \ge 3\).
By Theorem~\ref{thm 41}, the representation $\zeta$ is equivalent to one of the eight
representations $\zeta_i$, $1 \le i \le 8$, described therein.
Then the following statements hold.
\begin{itemize}
    \item[(a)] If \(\zeta\) is equivalent to one of $\zeta_1, \zeta_2,\zeta_3,\zeta_4,\zeta_5,\zeta_6$  or $\zeta_7$, then \(\zeta\) is  unfaithful.
    \item[(b)] If \(\zeta\) is equivalent to $\zeta_8$ and $a\in \{-1,0,1\}$, then \(\zeta\) is unfaithful.
    \end{itemize}
\end{theorem}

\begin{proof}
We deal with each case separately.
\begin{itemize}
\item[(a)] Suppose that $\zeta$ is equivalent to one of $\zeta_1, \zeta_2,\zeta_3,\zeta_4,\zeta_5,\zeta_6$  or $\zeta_7$. 
\begin{itemize}
    \item[(i)] If $\zeta \simeq \zeta_1$ or $\zeta_2$, then clearly $\zeta$ is unfaithful as the nontrivial elements $s_i$ of $M_kVT_n$ belong to $\ker{\zeta}$ for all $1\leq i \leq n-1$.
    \item[(ii)] If $\zeta \simeq \zeta_3,\zeta_4$ or $\zeta_5$, then direct matrix multiplications imply that the nontrivial elements \(s_i s_{i+1} s_i s_{i+1}\) of $M_kVT_n$ belong to \(\ker(\zeta)\) for all $1\leq i \leq n-2$. Consequently, \(\zeta\) is unfaithful. 
    \item[(iii)]  If \(\zeta \simeq\) $\zeta_6$, then the nontrivial elements \(s_i \rho_{i+1}^\alpha s_i \rho_{i+1}^\alpha s_i \rho_{i+1}^\alpha s_i \rho_{i+1}^\alpha\) of $M_kVT_n$ belong to \(\ker(\zeta)\) for all $1\leq i \leq n-2$ and $0\leq \alpha \leq k-1$. Hence, \(\zeta\) is unfaithful.
    \item[(iv)] If \(\zeta\simeq \) \(\zeta_7\), then the nontrivial elements \(s_{i+1} \rho_i^\alpha s_{i+1} \rho_i^\alpha s_{i+1} \rho_i^\alpha s_{i+1} \rho_i^\alpha \) of $M_kVT_n$ belong to \(\ker(\zeta)\) for all $1\leq i \leq n-2$ and $0\leq \alpha \leq k-1$. Thus, $\zeta$ is unfaithful.
\end{itemize}
\item[(b)] Suppose that \(\zeta\) is equivalent to $\zeta_8$. We consider each case separately.
\begin{itemize}
    \item[(i)] If $a=-1$, then the nontrivial elements $s_i \rho_{i+1}^\alpha s_i \rho_{i+1}^\alpha s_i \rho_{i+1}^\alpha s_i \rho_{i+1}^\alpha$ of $M_kVT_n$ belong to $\ker(\zeta)$ for all $1\leq i \leq n-2$ and $0\leq \alpha \leq k-1$. Hence, $\zeta$ is unfaithful.
    \item[(ii)] If $a=0$, then the nontrivial elements \(s_i \rho_{i+1}^\alpha s_i \rho_{i+1}^\alpha s_i \rho_{i+1}^\alpha\) of $M_kVT_n$ belong to \(\ker(\zeta)\) for all $1\leq i \leq n-2$ and $0\leq \alpha \leq k-1$, and so \(\zeta\) is unfaithful.
    \item[(iii)] If $a=1$, then the nontrivial elements $s_{i+1}\rho_i^\alpha s_{i+1}\rho_i^\alpha s_{i+1}\rho_i^\alpha s_{i+1}\rho_i^\alpha$ of $M_kVT_n$ belong to $\ker(\zeta)$ for all $1\leq i \leq n-2$ and $0\leq \alpha \leq k-1$ and so \(\zeta\) is again unfaithful.
\end{itemize}
\end{itemize}
\end{proof}

\begin{question}
A natural question arises regarding the faithfulness of $\zeta_8$: for $a \neq 0, \pm 1$, is the representation $\zeta_8$ faithful?
\end{question}


We now investigate the irreducibility of all complex homogeneous
$2$-local representations of $M_kVT_n$ for $k\geq 1$ and $n \ge 3$. 

\begin{theorem}
Let $
\zeta : M_kVT_n \longrightarrow \mathrm{GL}_n(\mathbb{C})$
be a complex homogeneous \(2\)-local representation of \(M_kVT_n\) with
$k\geq 1$ and $n \ge 3$. By Theorem~\ref{thm 41}, the representation $\zeta$ is equivalent to one of the eight
representations $\zeta_i$, $1 \le i \le 8$, described therein.
The reducibility of $\zeta$ is characterized as follows.
\begin{itemize}
    \item[(a)] If $\zeta$ is equivalent to $\zeta_1$, then $\zeta$ is reducible.
    \item[(b)] If $\zeta$ is equivalent to one of $\zeta_2, \zeta_3, \zeta_4,$ or $\zeta_5$,
    then $\zeta$ is reducible if and only if
    $y_{\beta_1}=y_{\beta_2}$ for all $0 \le \beta_1, \beta_2 \le k-1$.
    \item[(c)] If $\zeta$ is equivalent to one of $\zeta_6$ or $\zeta_7$,
    then $\zeta$ is irreducible.
    \item[(d)] If $\zeta$ is equivalent to $\zeta_8$, then $\zeta$ is reducible
    if and only if $y_{\beta_1}=y_{\beta_2}:=y$ for all $0 \le \beta_1, \beta_2 \le k-1$ and $\left( \dfrac{b}{y}=1+a \text{ \ or\ \  } \dfrac{b}{y}=1-a\right).$
\end{itemize}
\end{theorem}

\begin{proof}
We analyze each case separately, omitting the case (a) as it is trivial.
\begin{itemize}
\item[(b)]
Assume that $\zeta$ is equivalent to one of
$\zeta_2, \zeta_3, \zeta_4,$ or $\zeta_5$.
It suffices to treat the case $\zeta \simeq \zeta_2$,
as the remaining cases follow by similar arguments.

For the necessary condition, 
suppose that there exist $0 \le \beta_1 \neq \beta_2 \le k-1$ such that
$y_{\beta_1} \neq y_{\beta_2}$.
Let $\zeta_2'$ be the representation equivalent to $\zeta_2$ defined by
\[
\zeta_2'=P^{-1}\zeta_2P,
\qquad
P=\mathrm{diag}\big(y_{\beta_1}^{1-n},\,y_{\beta_1}^{2-n},\ldots,y_{\beta_1}^{-1},1\big).
\]
A direct computation shows that, for the generators $\rho_i^{\beta_1}$ and $\rho_i^{\beta_2}$ with $1\le i\le n-1$, the images under $\zeta_2'$ are given by $$\zeta_2'(\rho_i^{\beta_1})=\left( \begin{array}{c|@{}c|c@{}} \begin{matrix} I_{i-1} \end{matrix} & \textbf{0} & \textbf{0} \\ \hline \textbf{0} &\hspace{0.2cm} \begin{matrix} 0 & 1\\ 1 & 0\\ \end{matrix} & \textbf{0} \\ \hline \textbf{0} & \textbf{0} & I_{n-i-1} \end{array} \right)$$
and
$$\zeta_2'(\rho_i^{\beta_2})=\left( \begin{array}{c|@{}c|c@{}} \begin{matrix} I_{i-1} \end{matrix} & \textbf{0} & \textbf{0} \\ \hline \textbf{0} &\hspace{0.2cm} \begin{matrix} 0 & \dfrac{y_{\beta_1}}{y_{\beta_2}}\\ \dfrac{y_{\beta_2}}{y_{\beta_1}} & 0\\ \end{matrix} & \textbf{0} \\ \hline \textbf{0} & \textbf{0} & I_{n-i-1} \end{array} \right).$$
Assume, toward a contradiction, that $\zeta_2$ is reducible.
Then $\zeta_2'$ is reducible and so it admits a nontrivial proper invariant subspace
$U \subset \mathbb{C}^n$.
Let $u=(u_1,u_2,\ldots,u_n)^T \in U$ be nonzero. For each $1 \le i \le n-1$, we have
\[
\zeta_2'(\rho_i^{\beta_1})u - u
= (u_{i+1}-u_i)(e_i-e_{i+1}) \in U.
\]
If $u_i=u_{i+1}$ for all $1\leq i\leq n-1$, then $u$ is a scalar multiple of
$(1,1,\ldots,1)^T$, which is impossible since
$U$ is invariant under $\zeta_2'(\rho_i^{\beta_2})$ and
$y_{\beta_1} \neq y_{\beta_2}$.
Hence, there exists $1 \le j \le n-1$ such that $u_j\neq u_{j+1}$, and so we get that $e_j-e_{j+1} \in U$.
By applying $\zeta_2'(\rho_{j+1}^{\beta_1})$
and $\zeta_2'(\rho_{j-1}^{\beta_1})$ to $e_j-e_{j+1}$, we obtain
$$
\zeta_2'(\rho_{j+1}^{\beta_1})(e_j - e_{j+1})-(e_j - e_{j+1}) = e_{j+1} - e_{j+2} \in U
$$ 
and
$$\zeta_2'(\rho_{j-1}^{\beta_1})(e_j - e_{j+1})-(e_j - e_{j+1}) = e_{j-1} - e_j \in U.
$$
Continuing in the process we get 
\[
e_i-e_{i+1} \in U \quad \text{for all } 1 \le i \le n-1.
\]
Thus, no standard basis vector $e_i$ lies in $U$; otherwise,
$U=\mathbb{C}^n$, a contradiction.
Next,
\[
\zeta_2'(\rho_1^{\beta_2})(e_1-e_2)
+ \dfrac{y_{\beta_2}}{y_{\beta_1}}(e_1-e_2)
= \left(\dfrac{y_{\beta_2}}{y_{\beta_1}}
- \dfrac{y_{\beta_1}}{y_{\beta_2}}\right)e_1 \in U.
\]
Since $e_1 \notin U$, it follows that
$y_{\beta_1}=\pm y_{\beta_2}$.
As $y_{\beta_1}\neq y_{\beta_2}$, we must have
$y_{\beta_1}=-y_{\beta_2}$.
In this case we have,
\[
\zeta_2'(\rho_1^{\beta_1})(e_2-e_3)
-\zeta_2'(\rho_1^{\beta_2})(e_2-e_3)=2e_1 \in U,
\]
which implies a contradiction since $e_1\notin U$. Thus, $\zeta_2$ is irreducible whenever
$y_{\beta_1}\neq y_{\beta_2}$ for some $0\leq \beta_1 \neq \beta_2\leq k$, as required.\vspace{0.1cm}

For the sufficient condition, suppose that we have $y_{\beta_1}=y_{\beta_2}$ for all $0\leq \beta_1,\beta_2\leq k-1$. Then the vector $(1,1,\ldots,1)^T$ is invariant under the action of $\zeta_2'(s_i)$ and $\zeta_2'(\rho_i^\alpha)$ for all $1\leq i \leq n$ and \(0 \le \alpha \le k-1\). Hence, $\zeta_2'$ admits a nontrivial invariant subspace and is therefore reducible. Thus, $\zeta_2$ is also reducible. \vspace{0.1cm}
\item[(c)]
Assume that $\zeta$ is equivalent to $\zeta_6$ or $\zeta_7$.
Without loss of generality, suppose that $\zeta \simeq \zeta_6$.
We distinguish the following two cases.
\begin{itemize}
\item[(i)]
There exist $0\leq \beta_1 \neq \beta_2 \leq k-1$ such that $y_{\beta_1} \neq y_{\beta_2}$.
In this situation, repeating the argument of case (b) shows that
$\zeta$ is irreducible.
\item[(ii)]
$y_{\beta_1} = y_{\beta_2} := y$ for all $0 \le \beta_1, \beta_2 \le k-1$.
In this case, we consider an equivalent representation defined as follows.
Let $\zeta_6'$ be the representation equivalent to $\zeta_6$ given by
\[
\zeta_6'=Q^{-1}\zeta_6Q,
\qquad
Q=\mathrm{diag}\big(y^{1-n},\,y^{2-n},\ldots,y^{-1},1\big).
\]
A direct computation shows that, for the generators $s_i$ and $\rho_i^\alpha$
with $1 \le i \le n-1$ and $0 \le \alpha \le k-1$, we have
\[
\zeta_6'(s_i)=
\left(
\begin{array}{c|@{}c|c@{}}
\begin{matrix} I_{i-1} \end{matrix} & \mathbf{0} & \mathbf{0} \\
\hline
\mathbf{0} &
\begin{matrix}
\ 1 & yz\\
\ 0 & -1
\end{matrix}
& \mathbf{0} \\
\hline
\mathbf{0} & \mathbf{0} & I_{n-i-1}
\end{array}
\right),
\]
and
\[
\zeta_6'(\rho_i^{\alpha})=
\left(
\begin{array}{c|@{}c|c@{}}
\begin{matrix} I_{i-1} \end{matrix} & \mathbf{0} & \mathbf{0} \\
\hline
\mathbf{0} &
\begin{matrix}
\ 0 & 1\\
\ 1 & 0
\end{matrix}
& \mathbf{0} \\
\hline
\mathbf{0} & \mathbf{0} & I_{n-i-1}
\end{array}
\right).
\]
Assume, toward a contradiction, that $\zeta_6$ is reducible.
Then $\zeta_6'$ is also reducible and so it admits a non-trivial invariant subspace
$U \subset \mathbb{C}^n$.
As in case (b), invariance under
$\zeta_6'(\rho_i^{\alpha})$, $0 \le \alpha \le k-1$,
implies that $e_i - e_{i+1} \in U$ for all $1 \le i \le n-1$.
Consequently, no standard basis vector $e_i$ belongs to $U$;
otherwise, $U = \mathbb{C}^n$, a contradiction. On the other hand, we have
\[
\zeta_6'(s_1)(e_2 - e_3) + (e_2 - e_3) = yz\, e_1,
\]
with $y \neq 0$ and $z \neq 0$, which implies that $e_1 \in U$.
Hence $U = \mathbb{C}^n$, a contradiction.
Therefore, $\zeta_6$ is irreducible.
\end{itemize}

\item[(d)] Suppose now that $\zeta$ is equivalent to $\zeta_8$.

For the necessary condition, assume that $\zeta_8$ is reducible. A work similar to that in case (b) implies that $y_{\beta_1} = y_{\beta_2} := y$ for all $0 \le \beta_1, \beta_2 \le k-1$. So, we consider an equivalent representation defined in a similar way as in case (c). Let $\zeta_8'$ be the representation equivalent to $\zeta_8$ given by
\[
\zeta_8'=Q^{-1}\zeta_8Q,
\qquad
Q=\mathrm{diag}\big(y^{1-n},\,y^{2-n},\ldots,y^{-1},1\big).
\]
A direct computation shows that, for the generators $s_i$ and $\rho_i^\alpha$
with $1 \le i \le n-1$ and $0 \le \alpha \le k-1$, the actions under $\zeta_8'$ are given as follows.
\[
\zeta_8'(s_i)=
\left(
\begin{array}{c|@{}c|c@{}}
\begin{matrix} I_{i-1} \end{matrix} & \mathbf{0} & \mathbf{0} \\
\hline
\mathbf{0} &
\begin{matrix}
\ -a & - \dfrac{(a^2-1)y}{b}\\
\ \dfrac{b}{y} & a
\end{matrix}
& \mathbf{0} \\
\hline
\mathbf{0} & \mathbf{0} & I_{n-i-1}
\end{array}
\right),
\]
and
\[
\zeta_8'(\rho_i^{\alpha})=
\left(
\begin{array}{c|@{}c|c@{}}
\begin{matrix} I_{i-1} \end{matrix} & \mathbf{0} & \mathbf{0} \\
\hline
\mathbf{0} &
\begin{matrix}
\ 0 & 1\\
\ 1 & 0
\end{matrix}
& \mathbf{0} \\
\hline
\mathbf{0} & \mathbf{0} & I_{n-i-1}
\end{array}
\right).
\]
As $\zeta_8$ is reducible, we have $\zeta'_8$ is also reducible, and so it admits a nontrivial invariant subspace
$U \subset \mathbb{C}^n$. Let $u=(u_1,u_2,\ldots,u_n)^T \in U$ be nonzero. For each $1 \le i \le n-1$, we have
\[
\zeta_2'(\rho_i^{\alpha})u - u
= (u_{i+1}-u_i)(e_i-e_{i+1}) \in U.
\]
We now consider two cases as follows.
\begin{itemize}
    \item[(i)] Suppose that there exists $1 \le j \le n-1$ such that that $u_j\neq u_{j+1}$, then we get, as in case (b), that
\[
e_i-e_{i+1} \in U \quad \text{for all } 1 \le i \le n-1.
\]
In this case, we can see that 
$$\zeta_8'(s_1)(e_1-e_2)+\zeta_8'(s_1)(e_2-e_3)-(e_2-e_3)+a(e_1-e_2)=\left( \dfrac{b}{y}-1-a \right)e_2\in U.$$
But $e_2$ can not be in $U$, otherwise we get $U=\mathbb{C}^n$, which is impossible. Hence, we get that $ \dfrac{b}{y}-1-a=0$, and so $ \dfrac{b}{y}=1+a$.
\item[(ii)] Suppose now $u_j= u_{j+1}$ for all $1\leq j\leq n-1$. Then $U$ is generated by the vector $(1,1,\ldots ,1)^T$. Similarly, as in (i), we get that $\dfrac{b}{y}=1-a.$
\end{itemize}
Thus, we get that either $ \dfrac{b}{y}=1+a$ or $\dfrac{b}{y}=1-a,$ as required.

For the sufficient condition, suppose that $y_{\beta_1}=y_{\beta_2}$ for all $0 \le \beta_1, \beta_2 \le k-1$. We consider two cases as follows.
\begin{itemize}
    \item[(i)] If $\dfrac{b}{y}=1+a$, then one readily verifies that the vector $(1,1,\ldots,1)^T$ is invariant on the matrices $\zeta_8'(s_i)$ and $\zeta_8'(\rho_i^{\alpha})$ by left multiplication, for all $1\leq i\leq n-1$ and $0\leq \alpha \leq k-1$.
    
    \item[(ii)] If $\dfrac{b}{y}=1-a$, then one readily verifies that the vector $(1,1,\ldots,1)^T$ is invariant on the matrices $\zeta_8'(s_i)$ and $\zeta_8'(\rho_i^{\alpha})$ by right multiplication, for all $1\leq i\leq n-1$ and $0\leq \alpha \leq k-1$.
\end{itemize}

Hence, $\zeta_8'$ admits a nontrivial invariant subspace in both cases and is therefore reducible. Thus, $\zeta_8$ is also reducible.
\end{itemize}
\end{proof}


\section{Complex Representations of $M_2VPT_3$}
In this section, we focus on constructing representations of the multi-virtual pure twin group $M_kVPT_n$. Due to the complexity of this group for general values of $k$ and $n$, a complete treatment becomes technically demanding. Therefore, we restrict our attention to the case $k = 2$ and $n = 3$. From Section 3, we can see that the group $M_2VPT_3$ is the subgroup of $M_2VT_3$ that is generated by the elements $$\lambda_{12}^0, \lambda_{13}^0, \lambda_{23}^0, \lambda_{12}^{1}, \lambda_{13}^{1}, \text{ and } \lambda_{23}^{1} $$ subject to the following relations:
\begin{align*}
    \lambda_{12}^1 \lambda_{32}^1 &= \lambda_{32}^1 \lambda_{12}^1,\\
    \lambda_{13}^1 \lambda_{23}^1 &= \lambda_{23}^1  \lambda_{13}^1,\\
     \lambda_{12}^1 \lambda_{13}^1 &= \lambda_{13}^1  \lambda_{12}^1,\\
      \lambda_{12}^1 \lambda_{13}^1 \lambda_{23}^1 &= \lambda_{23}^1 \lambda_{13}^1 \lambda_{12}^1,\\
    \lambda_{12}^1 \lambda_{13}^1 \lambda_{23}^0 &= \lambda_{23}^0 \lambda_{13}^1 \lambda_{12}^1,\\
     \lambda_{21}^1 \lambda_{23}^1 \lambda_{13}^0 &= \lambda_{13}^0 \lambda_{23}^1 \lambda_{21}^1,\\
     \lambda_{13}^1 \lambda_{12}^1 \lambda_{32}^0 &= \lambda_{32}^0 \lambda_{12}^1 \lambda_{13}^1,\\
     \lambda_{23}^1 \lambda_{21}^1 \lambda_{31}^0 &= \lambda_{31}^0 \lambda_{21}^1 \lambda_{23}^1,\\
           \lambda_{31}^1 \lambda_{32}^1 \lambda_{12}^0 &= \lambda_{12}^0 \lambda_{32}^1 \lambda_{31}^1,\\
       \lambda_{32}^1 \lambda_{31}^1 \lambda_{21}^0 &= \lambda_{21}^0 \lambda_{31}^1 \lambda_{32}^1.
     \end{align*}
Moreover, these generators can be expressed in terms of the generators of $M_2VT_3$ in the following way:
$$\lambda_{12}^0=\rho_1s_1,\quad \lambda_{23}^0=\rho_2s_2,\quad 
\lambda_{13}^0=\rho_2\lambda_{12}^0\rho_2=\rho_2\rho_1s_1\rho_2,$$
$$\lambda_{12}^1=\rho_1\rho_1^1,\quad \lambda_{23}^1=\rho_2\rho_2^1,\quad \lambda_{13}^1=\rho_2\lambda_{12}^1\rho_2=\rho_2\rho_1\rho_1^1\rho_2.$$

In the remainder of this section, we investigate families of representations of $M_2VPT_3$ and analyze their irreducibility. We begin by recalling the results established in Section~4 regarding $2$-local representation of $M_kVT_n$, specialized to the case $k = 2$ and $n = 3$. 

\begin{definition}  \label{thm 51}
The following are the non-trivial complex homogeneous  $2$-local representations of $M_2VT_3$ that are classified in Theorem \ref{thm 41}.
$$s_1\mapsto \left(
\begin{array}{ccc}
 t & u & 0 \\
 v & w & 0 \\
 0 & 0 & 1 \\
\end{array}
\right),\quad 
s_2\mapsto \left(
\begin{array}{ccc}
 1 & 0 & 0 \\
 0 & t & u \\
 0 & v & w \\
\end{array}
\right),\quad 
\rho^0_1\mapsto \left(
\begin{array}{ccc}
 0 & \dfrac{1}{y_0} & 0 \\
 y_0 & 0 & 0 \\
 0 & 0 & 1 \\
\end{array}
\right),$$
$$\rho^0_2\mapsto 
\left(
\begin{array}{ccc}
 1 & 0 & 0 \\
 0 & 0 & \dfrac{1}{y_0} \\
 0 & y_0 & 0 \\
\end{array}
\right), \quad
\rho^1_1\mapsto 
\left(
\begin{array}{ccc}
 0 & \dfrac{1}{y_1} & 0 \\
 y_1 & 0 & 0 \\
 0 & 0 & 1 \\
\end{array}
\right),~\text{and}~
\rho^1_2\mapsto
\left(
\begin{array}{ccc}
 1 & 0 & 0 \\
 0 & 0 & \dfrac{1}{y_1} \\
 0 & y_1 & 0 \\
\end{array}
\right),$$
where the complex numbers $t,u,v,w,y_0,$ and $y_1$ satisfy one of the following cases.
\begin{enumerate}
    \item $\{t,w\}\subset\{1,-1\},$ $u=v=0$, $y_0\neq0,$ and $y_1\neq0$.
    \item $t=-w=\pm 1, u=z , v=0,$ $y_0\neq0,$ and $y_1\neq0$.
    \item $t=-a, u=-\dfrac{a^2-1}{b}, v=b, w=a,$ $y_0\neq0,$ and $y_1\neq0$.
\end{enumerate}
\end{definition}

By direct computations, we calculate the images of the generators of $M_2VPT_3$ under the representations given in Definition  \ref{thm 51}. In the following definition, we present some families of representations of $M_2VPT_3$.

\begin{definition}
 The representations of $M_2VPT_3$ that are obtained from the representations of $M_2VT_3$ given by Definition \ref{thm 51} are as follows.
\begin{enumerate}

\item $\lambda^0_{12}\mapsto\left(
\begin{array}{ccc}
 0 & \epsilon\dfrac{1}{y_0} & 0 \\
 \delta y_0 & 0 & 0 \\
 0 & 0 & 1 \\
\end{array}
\right), \quad
\lambda^0_{23}\mapsto
\left(
\begin{array}{ccc}
 1 & 0 & 0 \\
 0 & 0 & \epsilon \dfrac{1}{y_0} \\
 0 & \delta y_0 & 0 \\
\end{array}
\right),$\\
$\lambda^0_{13}\mapsto
\left(
\begin{array}{ccc}
 0 & 0 & \epsilon\dfrac{1}{y_0^2} \\
 0 & 1 & 0 \\
\delta y_0^2 & 0 & 0 \\
\end{array}
\right), \quad
\lambda^1_{12}\mapsto
\left(
\begin{array}{ccc}
 \dfrac{y_1}{y_0} & 0 & 0 \\
 0 & \dfrac{y_0}{y_1} & 0 \\
 0 & 0 & 1 \\
\end{array}
\right),$\\
$\lambda^1_{23}\mapsto
\left(
\begin{array}{ccc}
 1 & 0 & 0 \\
 0 & \dfrac{y_1}{y_0} & 0 \\
 0 & 0 & \dfrac{y_0}{y_1} \\
\end{array}
\right), \quad
\lambda^1_{13}\mapsto
\left(
\begin{array}{ccc}
 \dfrac{y_1}{y_0} & 0 & 0 \\
 0 & 1 & 0 \\
 0 & 0 & \dfrac{y_0}{y_1} \\
\end{array}
\right),$\\\\
where $y_0$ and $y_1$ are non-zero complex numbers and $\{\epsilon,\delta\}\subset\{-1,+1\}$.\\
\item $\lambda^0_{12}\mapsto
\left(
\begin{array}{ccc}
 0 & \mp\dfrac{1}{y_0} & 0 \\
\pm y_0 & y_0 z & 0 \\
 0 & 0 & 1 \\
\end{array}
\right), \quad
\lambda^0_{23}\mapsto
\left(
\begin{array}{ccc}
 1 & 0 & 0 \\
 0 & 0 & \mp\dfrac{1}{y_0} \\
 0 & \pm y_0 & y_0 z \\
\end{array}
\right),$ \\
$\lambda^0_{13}\mapsto
\left(
\begin{array}{ccc}
 0 & 0 & \mp\dfrac{1}{y_0^2} \\
 0 & 1 & 0 \\
\pm y_0^2 & 0 & y_0 z \\
\end{array}
\right), \quad
\lambda^1_{12}\mapsto
\left(
\begin{array}{ccc}
 \dfrac{y_1}{y_0} & 0 & 0 \\
 0 & \dfrac{y_0}{y_1} & 0 \\
 0 & 0 & 1 \\
\end{array}
\right),$\\
$\lambda^1_{23}\mapsto
\left(
\begin{array}{ccc}
 1 & 0 & 0 \\
 0 & \dfrac{y_1}{y_0} & 0 \\
 0 & 0 & \dfrac{y_0}{y_1} \\
\end{array}
\right), \quad \quad ~
\lambda^1_{13}\mapsto
\left(
\begin{array}{ccc}
 \dfrac{y_1}{y_0} & 0 & 0 \\
 0 & 1 & 0 \\
 0 & 0 & \dfrac{y_0}{y_1} \\
\end{array}
\right),$\\\\
where $y_0$, $y_1$,  and $z$ are non-zero complex numbers.\\
\item
$\lambda^0_{12}\mapsto\left(
\begin{array}{ccc}
 \dfrac{b}{y_0} & \dfrac{a}{y_0} & 0 \\
 -ay_0 & \dfrac{\left(1-a^2\right) y_0}{b} & 0 \\
 0 & 0 & 1 \\
\end{array}
\right), \quad
\lambda^0_{23}\mapsto\left(
\begin{array}{ccc}
 1 & 0 & 0 \\
 0 & \dfrac{b}{y_0} & \dfrac{a}{y_0} \\
 0 & -a y_0 & \dfrac{\left(1-a^2\right) y_0}{b} \\
\end{array}
\right),\\
\lambda^0_{13}\mapsto
\left(
\begin{array}{ccc}
 \dfrac{b}{y_0} & 0 & \dfrac{a}{y_0^2} \\
 0 & 1 & 0 \\
 -a y_0^2 & 0 & \dfrac{\left(1-a^2\right) y_0}{b} \\
\end{array}
\right), \quad
\lambda^1_{12}\mapsto
\left(
\begin{array}{ccc}
 \dfrac{y_1}{y_0} & 0 & 0 \\
 0 & \dfrac{y_0}{y_1} & 0 \\
 0 & 0 & 1 \\
\end{array}
\right),\\
\lambda^1_{23}\mapsto
\left(
\begin{array}{ccc}
 1 & 0 & 0 \\
 0 & \dfrac{y_1}{y_0} & 0 \\
 0 & 0 & \dfrac{y_0}{y_1} \\
\end{array}
\right), \hspace{2.1cm}
\lambda^1_{13}\mapsto
\left(
\begin{array}{ccc}
 \dfrac{y_1}{y_0} & 0 & 0 \\
 0 & 1 & 0 \\
 0 & 0 & \dfrac{y_0}{y_1} \\
\end{array}
\right)$,\\\\
where $a\in\mathbb{C}$, and  $y_0$, $y_1$, and $b$ are non-zero complex numbers.
\end{enumerate}
\end{definition}

\begin{theorem}
Consider the representation $\zeta: M_2VPT_3\longrightarrow \mathrm{GL}_{3}(\mathbb{C}) $  defined by
$$\lambda^0_{12}\mapsto\left(
\begin{array}{ccc}
 0 & \dfrac{1}{y_0} & 0 \\
 y_0 & 0 & 0 \\
 0 & 0 & 1 \\
\end{array}
\right), \quad
\lambda^0_{23}\mapsto\left(
\begin{array}{ccc}
 1 & 0 & 0 \\
 0 & 0 & \dfrac{1}{y_0} \\
 0 & y_0 & 0 \\
\end{array}
\right), \quad
\lambda^0_{13}\mapsto\left(
\begin{array}{ccc}
 0 & 0 & \dfrac{1}{y_0^2} \\
 0 & 1 & 0 \\
 y_0^2 & 0 & 0 \\
\end{array}
\right),$$
$$\lambda^1_{12}\mapsto\left(
\begin{array}{ccc}
 \dfrac{y_1}{y_0} & 0 & 0 \\
 0 & \dfrac{y_0}{y_1} & 0 \\
 0 & 0 & 1 \\
 
\end{array}
\right),\quad
\lambda^1_{23}\mapsto\left(
\begin{array}{ccc}
 1 & 0 & 0 \\
 0 & \dfrac{y_1}{y_0} & 0 \\
 0 & 0 & \dfrac{y_0}{y_1} \\
\end{array}
\right),\quad
\lambda^1_{13}\mapsto\left(
\begin{array}{ccc}
 \dfrac{y_1}{y_0} & 0 & 0 \\
 0 & 1 & 0 \\
 0 & 0 & \dfrac{y_0}{y_1} \\
\end{array}
\right).$$
Then $\zeta$ is irreducible if and only if $y_0\neq y_1$ .
    
\end{theorem}

\begin{proof}
The eigenvalues of $\lambda^1_{12}$ are $\dfrac{y_0}{y_1}, 1, \dfrac{y_1}{y_0}$, and the corresponding eigenvectors are $e_1=(1,0,0), e_2=(0,1,0), e_3=(0,0,1)$ respectively.
Note that $\lambda^1_{12}$, $\lambda^1_{23}$, and $\lambda^1_{13}$ have the same eigenvalues and eigenvectors. We distinguish three cases:
\begin{itemize}
    \item[(i)] $y_0^2\neq y_1^2$, the case of 3 distinct eigenvalues,
    \item[(ii)] $y_0=-y_1$, the case of 2 equal eigenvalues,
    \item[(iii)] $y_0=y_1$, the case of 3 equal eigenvalues.
\end{itemize}
By direct computations, we obtain
$$\zeta(\lambda^0_{12})e_1=y_0e_2, ~\zeta(\lambda^0_{23})e_2=y_0e_3,$$ $$\zeta(\lambda^0_{23})e_3=\dfrac{1}{y_0}e_2, \text{ and } \zeta(\lambda^0_{12})e_2=\dfrac{1}{y_0}e_1.$$

This shows that neither the one-dimensional subspaces $\langle e_i\rangle$ nor the two-dimensional subspaces $\langle e_i,e_j\rangle$ are invariant under $\zeta$ for $1\leq i,j\leq 3$. Therefore, in case (i), the representation $\zeta$ is irreducible.

For case (ii), the eigenvalues of $\lambda^1_{12}$ are $-1, -1$ and $1$ with eigenvectors $e_1,e_2$ and $e_3$.
In this case, it remains to check whether the subspaces $\langle e_1+xe_2\rangle$ and $\langle e_1+xe_2, e_3\rangle $ are invariant for $x\in\mathbb{C}\setminus\{0\}$. Note that $\zeta(\lambda^0_{13})(e_1+xe_2)=xe_2+y_1^2e_3\not\in\langle e_1+xe_2\rangle$. This implies that $\langle e_1+xe_2\rangle$ is not invariant under $\zeta$ for all $x\in\mathbb{C}\setminus\{0\}.$ Let $S$ be a non-trivial invariant subspace of $\mathbb{C}^3$ spanned by the vectors $e_1+xe_2$ and $e_3$. Since $\zeta(\lambda^0_{23})(e_3)=-\dfrac{1}{y_1}e_2$, it follows that $e_2\in S$.
Then $e_1\in S$ because $e_1=e_1+xe_2-xe_2$. Thus, $S=\mathbb{C}^3$ as it contains $e_1, e_2$ and $e_3$.
This contradicts the fact that $S$ is a non-trivial subspace. Hence, there are no non-trivial invariant subspaces of the form $\langle e_1+xe_2, e_3\rangle$.
Therefore, in case (ii), the representation $\zeta$ is irreducible.

In case (iii), we can easily check that the subspace spanned by the vector $e_1+y_1e_2+y_1^2e_3$ is an invariant subspace under $\zeta$. Hence, $\zeta$, in case (iii) is reducible. This completes the proof.
\end{proof}

\begin{theorem}
Consider the representation $\zeta: M_2VPT_3\longrightarrow \mathrm{GL}_{3}(\mathbb{C}) $  defined by

$$\lambda^0_{12}\mapsto\left(
\begin{array}{ccc}
 \dfrac{b}{y_0} & \dfrac{a}{y_0} & 0 \\
 -ay_0 & \frac{\left(1-a^2\right) y_0}{b} & 0 \\
 0 & 0 & 1 \\
\end{array}
\right),\quad
\lambda^0_{23}\mapsto\left(
\begin{array}{ccc}
 1 & 0 & 0 \\
 0 & \dfrac{b}{y_0} & \dfrac{a}{y_0} \\
 0 & -a y_0 & \frac{\left(1-a^2\right) y_0}{b} \\
\end{array}
\right),$$
$$\lambda^0_{13}\mapsto
\left(
\begin{array}{ccc}
 \dfrac{b}{y_0} & 0 & \dfrac{a}{y_0^2} \\
 0 & 1 & 0 \\
 -a y_0^2 & 0 & \frac{\left(1-a^2\right) y_0}{b} \\
\end{array}
\right), \quad
\lambda^1_{12}\mapsto
\left(
\begin{array}{ccc}
 \dfrac{y_1}{y_0} & 0 & 0 \\
 0 & \dfrac{y_0}{y_1} & 0 \\
 0 & 0 & 1 \\
\end{array}
\right),$$
$$\lambda^1_{23}\mapsto
\left(
\begin{array}{ccc}
 1 & 0 & 0 \\
 0 & \dfrac{y_1}{y_0} & 0 \\
 0 & 0 & \dfrac{y_0}{y_1} \\
\end{array}
\right),\quad 
\lambda^1_{13}\mapsto
\left(
\begin{array}{ccc}
 \dfrac{y_1}{y_0} & 0 & 0 \\
 0 & 1 & 0 \\
 0 & 0 & \dfrac{y_0}{y_1} \\
\end{array}
\right).$$

Suppose that $y_0\neq y_1$. 
Then $\zeta$  is irreducible if and only if $a\neq 0$.
     
\end{theorem}
\begin{proof}
The eigenvalues of $\lambda^1_{12}$ are: $\dfrac{y_0}{y_1}, \dfrac{y_1}{y_0}, 1$, and the corresponding eigenvectors are $e_1=(1,0,0),  e_2=(0,1,0),  e_3=(0,0,1)$ respectively.
Note that $\lambda^1_{12}$, $\lambda^1_{23}$ and $\lambda^1_{13}$ have the same eigenvalues and eigenvectors.
Direct computations yield the following:
$$\zeta(\lambda^0_{12})(e_1)=\dfrac{b}{y_0}e_1-ay_0e_2,\quad \zeta(\lambda^0_{13})(e_1)=\dfrac{b}{y_0}e_1-ay_0^2e_3, \text{ and } \zeta(\lambda^0_{23})(e_1)=e_1.$$

First, suppose that $a=0$. Then the subspace spanned by the vector $e_1=(1,0,0)$ is invariant. Hence $\zeta$ is reducible if $a=0$.

Now, suppose that $a\neq 0$. We have two cases: \\
(i) $y_0^2\neq y_1^2$ (3 distinct eigenvalues of $\zeta(\lambda_{12}^1)$),\\
(ii) $y_0=-y_1$ (2 equal eigenvalues).

Because we have
$$\zeta(\lambda^0_{23})(e_2)=-\dfrac{b}{y_0}e_2-ay_0e_3, \quad \zeta(\lambda^0_{13})(e_3)=\dfrac{a}{y_0^2}e_1+\dfrac{(1-a^2)y_0}{b}e_3,$$
$$\zeta(\lambda^0_{23})(e_3)=\dfrac{a}{y_0}e_2+\dfrac{(1-a^2)y_0}{b}e_3, \quad \zeta(\lambda^0_{12})(e_2)=\dfrac{a}{y_0}e_1+\dfrac{(1-a^2)y_0}{b}e_3,$$ and
$$\zeta(\lambda^0_{13})(e_1)=\dfrac{b}{y_0}e_1-ay_0^2e_3,$$
it follows that the subspaces  $\langle e_i\rangle$ and the subspaces $\langle e_i,e_j\rangle$ are not invariant.
Hence, in case (i) $y_0^2\neq y_1^2$, the representation $\zeta$ is irreducible if $a\neq 0$.

In case (ii) $y_0=-y_1$, we have to consider the subspaces 
spanned by 
$\langle e_1+xe_2 \rangle $ and the subspaces $\langle e_1+xe_2,e_3\rangle $ for $x\in\mathbb{C}\setminus\{0\}$. Because we have 
$$ \zeta(\lambda_{23}^0)(e_1+xe_2)=e_1-\dfrac{bx}{y_1}e_2+ay_1xe_3 \text{ and } \zeta(\lambda_{13}^0)e_3=\dfrac{a}{y_1^2}e_1-\dfrac{\left(1-a^2\right)y_1}{b}e_3,$$
it follows that the subspaces spanned by $e_1+xe_2$ and the subspaces $\langle e_1+xe_2,e_3\rangle$ are not invariant for all $x\in\mathbb{C}\setminus\{0\}$. Hence, $\zeta$ is irreducible in case $y_0=-y_1$, and this completes the proof.
\end{proof}
\section{Conclusion}

Since the construction of these groups is new, several natural directions for future research arise. In particular, it would be of interest to introduce the multi-welded twin group, analogous to the multi-welded braid group. Another promising direction is to construct representations of the multi-virtual semi-pure twin group $M_kVHT_n$, in a manner similar to the representations of $M_kVPT_n$, and to study their algebraic properties. Furthermore, it is desirable to generalize the results obtained for $M_2VPT_3$ to arbitrary values of $k$ and $n$, and to investigate whether similar techniques and conclusions extend to the general case.

\section{Acknowledgment}
The first author acknowledges the support of the University Grants Commission
(UGC), India, for a research fellowship with NTA Ref. No.231610035955. The third author acknowledges the support of the Anusandhan National Research
Foundation (ANRF) with sanction order no. CRG/2023/004921.

\end{document}